\documentclass[11pt]{amsart}

\usepackage{amsmath}
\usepackage{amssymb}
\usepackage{amsfonts}
\usepackage{mathtools}
\usepackage{enumerate}
\usepackage{comment}
\allowdisplaybreaks
\usepackage{palatino}
\usepackage[T1]{fontenc}
\usepackage[mathscr]{eucal}
\usepackage{dsfont}
\usepackage{fullpage}

\usepackage{acronym}
\usepackage{latexsym}
\usepackage{paralist}
\usepackage{wasysym}
\usepackage{xspace}

\usepackage[dvipsnames,svgnames]{xcolor}
\colorlet{MyBlue}{DodgerBlue!60!Black}
\colorlet{MyGreen}{DarkGreen!85!Black}

\usepackage[font=small,labelfont=bf]{caption}
\usepackage{subfigure}
\usepackage{tikz}
\usetikzlibrary{calc,patterns}
\usepackage[square,sort,comma,numbers]{natbib}

\usepackage{hyperref}
\hypersetup{
	colorlinks=true,
	linktocpage=true,
	pdfstartview=FitH,
	breaklinks=true,
	pdfpagemode=UseNone,
	pageanchor=true,
	pdfpagemode=UseOutlines,
	plainpages=false,
	bookmarksnumbered,
	bookmarksopen=false,
	bookmarksopenlevel=1,
	hypertexnames=true,
	pdfhighlight=/O,
	urlcolor=MyBlue!60!black,linkcolor=MyBlue!70!black,citecolor=DarkGreen!70!black, % <--- for screen
	pdftitle={},
	pdfauthor={},
	pdfsubject={},
	pdfkeywords={},
	pdfcreator={pdfLaTeX},
	pdfproducer={LaTeX with hyperref}
}

\numberwithin{equation}{section}  %numberwithin goes before cleverefs when using hyperref
\usepackage[sort&compress,capitalize,nameinlink]{cleveref}
\crefname{app}{Appendix}{Appendices}

\crefrangeformat{equation}{\upshape(#3#1#4)\textendash(#5#2#6)}

\usepackage[textwidth=30mm]{todonotes}

\usepackage{soul}
\setstcolor{red}
\sethlcolor{SkyBlue}

\newcommand{\debug}[1]{{\color{black}#1}}
\theoremstyle{plain}
\newtheorem{theorem}{Theorem}
\newtheorem{corollary}[theorem]{Corollary}
\newtheorem*{corollary*}{Corollary}
\newtheorem{lemma}[theorem]{Lemma}
\newtheorem{proposition}[theorem]{Proposition}
\newtheorem{conjecture}[theorem]{Conjecture}

\theoremstyle{definition}

\newtheorem*{definition*}{Definition}

\newtheorem*{hypothesis*}{Hypothesis}

\theoremstyle{remark}
\newtheorem{remark}[theorem]{Remark}
\newtheorem*{remark*}{Remark}
\newtheorem*{notation*}{Notational remark}

\numberwithin{theorem}{section}
\DeclarePairedDelimiter{\norm}{\lVert}{\rVert}

\DeclarePairedDelimiter{\ceil}{\lceil}{\rceil}

\DeclarePairedDelimiterX{\braket}[2]{\langle}{\rangle}{#1,#2}

\DeclarePairedDelimiterX{\inner}[2]{\langle}{\rangle}{#1,#2}
\DeclarePairedDelimiterX{\setdef}[2]{\{}{\}}{#1:#2}

\DeclarePairedDelimiterXPP{\probof}[1]{\Prob}{(}{)}{}{%
	
	#1}

\DeclarePairedDelimiterXPP{\exof}[1]{\Expect}{[}{]}{}{%
	
	#1}

\newcommand{\abs}[1]{\left|#1\right|}
\newcommand{\tonde}[1]{\left(#1\right)}
\newcommand{\quadre}[1]{\left[#1\right]}
\newcommand{\ttonde}[1]{\big(#1\big)}

\newcommand{\emparg}{\,\cdot\,}
\newcommand{\emp}{\varnothing}
\newcommand{\eqdef}{\coloneqq}

\newcommand{\car}{\mathds{1}}

\renewcommand{\complement}{c}

\newcommand{\fstop}{\; \text{.}}
\newcommand{\comma}{\; \text{,}\;\;}

\newcommand{\eps}{\varepsilon}

\newcommand{\Expect}{\mathbf{\debug E}}
\newcommand{\Prob}{\mathbf{\debug P}}

\newcommand{\ind}{\mathds{1}}

\newcommand{\cA}{\ensuremath{\mathcal A}} 
 
\newcommand{\cC}{\ensuremath{\mathcal C}} 
\newcommand{\cD}{\ensuremath{\mathcal D}} 
\newcommand{\cE}{\ensuremath{\mathcal E}} 
\newcommand{\cF}{\ensuremath{\mathcal F}} 
\newcommand{\cG}{\ensuremath{\mathcal G}} 
\newcommand{\cH}{\ensuremath{\mathcal H}} 
 
\newcommand{\cJ}{\ensuremath{\mathcal J}} 
 
\newcommand{\cL}{\ensuremath{\mathcal L}} 
\newcommand{\cM}{\ensuremath{\mathcal M}} 
\newcommand{\cN}{\ensuremath{\mathcal N}}

\newcommand{\cQ}{\ensuremath{\mathcal Q}}

\newcommand{\cW}{\ensuremath{\mathcal W}}

\newcommand{\E}{\ensuremath{\mathbb{E}}}

\newcommand{\N}{\ensuremath{\mathbb{N}}}

\newcommand{\Q}{\ensuremath{\mathbb{Q}}}

\renewcommand{\P}{\ensuremath{\mathbb{P}}}

\newcommand{\Bin}{\text{\normalfont Bin}}
\newcommand{\img}{\text{\normalfont Im}}

\newcommand{\clr}{c}
\def\({\left(}
\def\){\right)}
\def\[{\left[}
\def\]{\right]}
\newacro{NE}{Nash equilibrium}
\newacroplural{NE}[NE]{Nash equilibria}
\newacro{PNE}{pure Nash equilibrium}
\newacroplural{PNE}[PNE]{pure Nash equilibria}
\newacro{MNE}{mixed Nash equilibrium}
\newacroplural{MNE}[MNE]{mixed Nash equilibria}
\newacro{PFNE}{prior-free Nash equilibrium}
\newacroplural{PFNE}[PFNE]{prior-free Nash equilibria}
\newacro{WE}{Wardrop equilibrium}
\newacroplural{WE}[WE]{Wardrop equilibria}
\newacro{SO}{socially optimum}
\newacro{SU}{social utility}
\newacro{BEq}{best equilibrium}
\newacro{WEq}{worst equilibrium}
\newacro{KKT}{Karush\textendash Kuhn\textendash Tucker}
\newacro{OD}[O/D]{origin-destination}
\newacro{PoA}{price of anarchy}
\newacro{PoS}{price of stability}
\newacro{PoCS}{price of correlated stability}
\newacro{BPR}{bureau of public roads}
\newacro{FIP}{finite improvement property}
\newacro{CLT}{central limit theorem}

\newacro{BPG}{buck-passing game}
\newacro{SBPG}{stochastic buck-passing game}
\newacro{MBPG}{mixed extension of the buck-passing game}

\begin{document}
	\title[On the meeting of random walks on random DFA]{On the meeting of random walks on random DFA}
	
	\author[M.~Quattropani]{Matteo Quattropani$^{\dagger}$}
	\address{$^{\dagger}$ Dipartimento di Matematica ``Guido Castelnuovo'', Sapienza Universit\`a di Roma, Piazzale Aldo Moro 5, 00185, Roma, Italy}
	\email{matteo.quattropani@uniroma1.it}
	\author[F.~Sau]{Federico Sau$^\star$}
	\address{$^\star$ Dipartimento di Matematica e Geoscienze, Universit\`a degli Studi di Trieste, Via Valerio 12/1, 34127, Trieste, Italy}
	\email{federico.sau@units.it}              % 
	\begin{abstract}
		We consider two random walks evolving synchronously on a random out-regular graph of $n$ vertices with bounded out-degree $r\ge 2$, also known as a random Deterministic Finite Automaton (DFA). We show that, with high probability with respect to the generation of the graph, the meeting time of the two walks is stochastically dominated by a geometric random variable of rate $(1+o(1))n^{-1}$, uniformly over their starting locations. Further, we prove that this upper bound is typically tight, i.e., it is also a lower bound when the locations of the two walks are selected uniformly at random. Our work takes inspiration from a recent conjecture by Fish and Reyzin \cite{FR2017} in the context of computational learning, the connection with which is discussed.
	\end{abstract}
\maketitle
	\section{Introduction}
	Since the seminal work of Cox \cite{cox_coalescing1989}, \emph{coalescing  random walks} has become  a classical subject in probability, the last decade, in particular, registering several 	important  developments. 
	In the reversible setting, for instance, the works \cite{cooper_frieze_radzik_2009,OtrAMS2012,cooper_elsasser_ono_radzik_2013,kanade_mallmann-trenn_sauerwald2019,oliveira_peres_2019} establish a number of estimates for the mean coalescing time in terms of meeting, hitting, returning, and relaxation times.
	In the more general context of non-reversible  random walks, the work by Oliveira  \cite{oliveira_mean2013} characterizes the limit distribution of the coalescence time under suitable mean field conditions. 
	Perhaps the most striking consequence of these  conditions is that they ensure that the timescale at which coalescence occurs coincides with that of the meeting time of two random walks starting from equilibrium. This result nearly solves  Open Problem 14.12 in \cite{aldous-fill-2014}, and reinforces the intuition that, in this context and on this timescale, the number of coalescing random walks must be well-approximated by the number of partitions in Kingman's coalescent (see  \cite{beltran_chavez_landim2019} and references therein).
	Moreover, such mean field conditions  are, on the one hand, easily verifiable in several concrete examples, as they involve estimates essentially only on  the mixing time and invariant measure  of the single walk; on the other hand, they are very general -- they  do not require  reversibility, for instance (cf.\ \cite[Theorem 1.2]{oliveira_mean2013}).

	The study  in \cite{oliveira_mean2013} provides a fairly general framework in which the connection between meeting and coalescence times is well-understood. However, in each of these situations, 
	extracting finer quantitative information on coalescence  must still necessarily go through the problem of quantitatively analyzing  the meeting of two walks. Solving the  latter  requires  \emph{ad hoc} analyses depending on the graph of interest, and, for random walks on random graphs, it has  been addressed only in the regular undirected  setting (\cite{cooper_frieze_radzik_2009}).
	
	\
	
	In this work, we  quantitatively analyze the \emph{meeting time} of two random walks on a model of  \emph{sparse random directed graphs}. Such random walks evolve independently, and, as most commonly done  in the theoretical computer science literature, we model  them to move in discrete synchronous rounds. The strategy that we adopt in our analysis is related to that in \cite{cooper_frieze_radzik_2009}, in which the authors are concerned, among other things, with analogous quantitative estimates for walks on random regular graphs. In our context, though, the directness of the graph is what makes the analysis much more involved. For instance, the stationary distribution of a sparse random digraph is a highly non-trivial random object, whose properties cannot be inferred from a local analysis of the  graph.
	
	\emph{Random walks on random directed graphs} is, in fact, an emerging topic in the field, with a number of advances in the last few years for what concerns  the study of  \textcolor{black}{total-variation} mixing times (\cite{BCSrwrd,bordenave_cutoff2019,caputo_quattropani_2021_RSA,caputo_quattropani_2021_SPA}) and  stationary distributions (\cite{addario-berry_diameter2020,caputo_quattropani_2020,CPminimum2020,CCPQ2021}).
	All these works deal with the behavior of a single walk, while the results in our paper represent a first step toward the analysis of multiple walks on these geometries.
	In particular, we prove that, with high probability with respect to the generation of the graph, any two walks meet at a time which is stochastically dominated by a geometric random variable of mean $(1+o(1))n$. Further, we establish that this upper bound is typically tight,  turning it into an effective lower bound for when the two walks are selected uniformly at random.
	Finally, our quantitative results also relate to some open problems  within the framework of learning and synchronizing random DFAs, two important topics in machine learning and automata theory. (We refer to \cref{sec:motivation} below for a more thorough discussion on this connection.) 
	
	The main technical tool in our proofs is the so-called \emph{First Visit Time Lemma} (FVTL), originally introduced by Cooper and Frieze in \cite{cooper_frieze_2004}, and recently reinterpreted by the authors of \cite{manzo_quattropani_scoppola_2021} within the framework of quasi-stationary distributions. The FVTL provides sharp asymptotic estimates for the tail probabilities of the hitting time of a given state of a Markov chain, when the process starts from stationarity. As in \cite{cooper_frieze_radzik_2009}, we recast the original \textquoteleft meeting problem\textquoteright\ for the two walks into a \textquoteleft hitting problem\textquoteright\ for the product chain, by considering all diagonal elements as merged so to form the single target state. The FVTL is then applied to a natural \emph{auxiliary chain} resulting from this procedure. In the undirected setting, this auxiliary chain is just the product chain in which all diagonal elements have been collapsed into a single vertex, retaining all the edges; clearly, this operation preserves the stationary distribution of all the off-diagonal states. 
	This strategy gets  more involved when the underlying graph is directed. We overcome this difficulty by adopting the generalization of the auxiliary chain  recently introduced in \cite{manzo_quattropani_scoppola_2021}, and  derive refined bounds for its stationary distribution and mixing times,  yielding sharp asymptotics for the meeting time of two independent walks.

	\
	
	The rest of the paper is organized as follows. In \cref{sec:results}, we present the model and the corresponding main results. In particular, in  \cref{sec:motivation}, we link our results to some open problems within the framework of learning and synchronizing random DFAs. In \cref{sec:preliminaries}, we introduce the auxiliary chain and state the FVTL. \cref{sec:meeting-stationarity} contains the main technical contribution of the paper, in which we establish the precise asymptotic distribution of the meeting time of two walks starting from stationarity. The proof of the latter is split into several lemmas, \textcolor{black}{and its organization is spelled out in detail in \cref{sec:organization}}. Finally, \cref{sec:theorems_proofs} is devoted to the proofs of our main results.   
	
	\section{Model, main results, and motivations}\label{sec:results}
	For $n, r \in \N\eqdef \{0,1,\ldots\}$ and $2 \le r\le n$, let
	\begin{align}\label{eq:V-C-f}
		V\eqdef [n]\eqdef \{1,\ldots, n\}\comma\qquad \cC\eqdef[r]\eqdef \{1,\ldots, r \}\comma\qquad \{f_x:\cC\to V\ \text{one-to-one}\}_{x\in V}\fstop
	\end{align}
	The triple $(V,\cC,\{f_x\}_{x\in V})$ is known  as a \emph{Deterministic Finite Automaton} (DFA) with  \emph{states} $V$ and \emph{alphabet} $\cC$. 	This can be equivalently represented as 
	a \emph{colored $r$-out regular graph}, where:
	\begin{itemize}
		\item $V$ is the vertex set;
		\item $\cC$ is the set of colors;
		\item $\img(f_x)\subset V$ are the $r$ out-neighbors of $x\in V$, with the directed edge $e=(x,f_x(c))$ uniquely endowed with the color $c\in \cC$.
	\end{itemize} In such a directed graph,  each vertex has one out-going edge for each color in $\cC$, possibly with self-loops, but with no multiple directed edges.

	Considering  \emph{random} mappings $\{f_x\}_{x\in V}$ gives rise to a  random realization $G=G^{(V,\cC)}$ of such an object, typically referred to as a \emph{random DFA}. In the language of  colored  graphs, this random construction goes as follows:	to each  $x\in V$, attach $r$ out-stubs (tails), one for each color in $\cC$, and  independently select $r$ elements in $V$ without replacement and attach to each of them a distinct colored \textcolor{black}{out-stub} of $x$. 	
	Note that such a random DFA is uniformly distributed over all possible  DFA with  states $V$ and alphabet $\cC$.
	
	Given a realization of a random DFA, the \emph{random walk on $G$} is the (discrete-time) Markov chain $(X_t)_{t\in \N} \in V^\N$, with laws $(\Prob_x)_{x\in V}$ such that $\Prob_x(X_0=x)=1$ induced by the transition matrix $P=P^{(G)}$ given by	
	\begin{align*}
		P(x,y)\eqdef \frac{1}{r}\sum_{\clr\in \cC} \car_{\{y\}}(f_x(\clr))\comma\qquad x , y \in V\fstop
	\end{align*}
	In words,  at each step, the  walk	 selects  uniformly at random a color $\clr\in \cC$ and follows the unique outgoing edge  having that color. Note that, for every $x\in V$, paths of length $t\in \N$ under $\Prob_x$ can be sampled by choosing uniformly at random an element of $\cC^t$. We will refer to an element $w \in \cC^t$ as a \emph{word of length $t$}.

	Our main results concern two such walks evolving  \emph{synchronously} and  \emph{independently}. This system of two walks corresponds to the product Markov chain $(\mathbf X^{(2)}_t)_{t\in \N}=(X^{(1)}_t, X^{(2)}_t)_{t\in \N} \in (V^\N)^2$ with laws $(\Prob_{(x,y)}=\Prob^{(2)}_{(x,y)})_{(x,y)\in V^2}$ induced by the transition matrix
	$P^{(2)}\eqdef P\otimes P$. In this case, for every $(x,y)\in V^2$, paths	 of length $t \in \N$ under $\Prob_{(x,y)}$ are sampled by choosing two \emph{independent} random words of length $t$.  For such a product chain, we refer to the following stopping time
	\begin{align}\label{eq:tau_meet}
		\tau_{\rm meet}\eqdef \inf\{t\in \N: X^{(1)}_t=X^{(2)}_t\}\comma 
	\end{align}
	as the \emph{meeting time} of the two walks.

	Our analysis is carried out in an asymptotic setting, in which the vertex set grows ($n\to \infty$), while the number of colors stays fixed ($r\in \N$, $r \ge 2$). As a consequence, $r$ is often omitted from the notation, and all the asymptotic notation refers (often implicitly) to the limit $n\to \infty$. Finally, the following notation will be used all throughout:
	\begin{itemize}
		\item $(\Omega,\cF,\P)=(\Omega^{(n)},\cF^{(n)},\P^{(n)})$ denotes the probability space of the random DFA $G=G^{(n)}$, with $\E=\E^{(n)}$ denoting the corresponding expectation.
		\item For two  sequences $Y=Y^{(n)}$ and $Z=Z^{(n)}$ of random variables (both measurable with respect to  the random DFA $G=G^{(n)}$), we write
		\begin{align*}
			Y\overset{\P}\longrightarrow Z\qquad \overset{\circ}\Longleftrightarrow\qquad \lim_{n\to \infty}\P\tonde{\abs{Y-Z}>\eps}=0\comma\qquad \eps >0\fstop
		\end{align*}
		\item For a sequence $\cE=\cE^{(n)}$ of events in $\Omega=\Omega^{(n)}$, \textquotedblleft $\cE$ occurs w.h.p.\textquotedblright\ if $\lim_{n\to \infty}\P(\cE)=1$.
	\end{itemize}
	We  now  present our main results.
	\begin{theorem}\label{th:1}
		There exist 	random variables $\Lambda=\Lambda^{(n)} \in (0,1)$ such that		
		\begin{align}\label{eq:Lambda}
			\Lambda\, n  \overset{\P}{\longrightarrow} 1\comma
		\end{align}
		and, for every $\eps >0$, w.h.p., 	
		\begin{align}\label{eq:th1-ub}
			\sup_{t\ge 0} \max_{x,y \in V}\frac{\Prob_{(x,y)}\tonde{\tau_{\rm meet}>t}}{(1-\Lambda)^t}<1+\eps\fstop
		\end{align}
	\end{theorem}
	
	In words, \cref{th:1} states that for a typical realization of a random DFA, uniformly over the starting positions of two independent walks, the tails of their meeting time are bounded above by those of a geometric random variable of mean $(1+o(1)) n$.

	As an improvement of this result,  we show that the upper bound in \cref{eq:th1-ub} is  tight for most couples $(x,y)$, $x\neq y$; this is the content of the following:
	\begin{theorem}\label{th:2}
		Recall $\Lambda=\Lambda^{(n)}$ from \cref{th:1}. Then, for any couple $(x,y)=(x^{(n)},y^{(n)})\in V^2$ of distinct states,
		\begin{align}
			\sup_{t\ge 0}\abs{\frac{\Prob_{(x,y)}\tonde{\tau_{\rm meet}>t}}{(1-\Lambda)^t}-1}\overset{\P}\longrightarrow 0\fstop
		\end{align}
	\end{theorem}
	
	As an immediate consequence of \cref{eq:Lambda} and \cref{th:2}, we get:
	\begin{corollary}\label{coro:mean-meeting}
		For any couple $(x,y)=(x^{(n)},y^{(n)})\in V^2$ of distinct states,
		\begin{align*}
			\frac{\Expect_{(x,y)}\quadre{\tau_{\rm meet}}}{n}\overset{\P}{\longrightarrow}	1\fstop
		\end{align*} 
	\end{corollary}
	
		\begin{color}{black}
			It is worth to remark that the distribution of the meeting times in \cref{th:1,th:2} does not depend on the choice of the out-degree $r$. We postpone a discussion on this point to \cref{rmk:indep-r-2}.
	\end{color}
	
	\subsection{Motivation and related open problems: reconstructing and synchronizing random DFAs}\label{sec:motivation}
	DFA is a classical model in the theory of computation (see, e.g., \cite{hopcroft2001introduction}), and its first appearance in the literature can be traced back to \cite{mcculloch1943logical}.  
	We  recall that, for a given  DFA $(V,\cC,\{f_x \}_{x\in V})$ (cf.\ \cref{eq:V-C-f}) and for every $t\in \N$, $\cC^t$ denotes the set of words of length $t$; further, for a given state $v$ and a word $w$ of finite length, then $w(v)$ indicates the state reached by following the letters of $w$ when starting from $v$.

	\subsubsection{Learning a  DFA, and meeting times}\label{sec:learning}
	Usually, a DFA is equipped with a special state $v$ called \emph{root} and a subset of \emph{accepting states} $F\subseteq V$, in which case one speaks about a (deterministic finite) \emph{acceptor} $(V,\cC,\{f_x \}_{x\in V},v,F)$.
	Acceptors constitute a very simple model of a finite-state machine that accepts or rejects a given word (of finite length) $w$ depending on whether $w(v)\in F$ or not. The set of all finite accepted words for a given acceptor is  referred to as the language recognized by the acceptor. A prominent problem in computational learning theory is that of \emph{reconstructing the language} of an underlying acceptor  given a  set of information provided by an oracle. Such learning problems, when associated to a \emph{worst case} underlying acceptor, are notoriously extremely hard to solve (see, e.g., \cite{angluin1981}). For this reason, part of the recent literature on the subject is devoted to an \emph{average case analysis}, in which the acceptor -- and, in particular, the associated DFA -- is chosen at random. 
	
	In the attempt to provide an efficient algorithm to learn a random acceptor, the authors in \cite{FR2017} propose an open problem that can be rephrased in terms of random walks on a random DFA. For a fixed $t\in\N$, let $\mathbf{Q}=\mathbf{Q}_t$ be the uniform distribution over $\cC^t$, and $W_t$ a random word sampled according to $\mathbf{Q}$.  Fish and Reyzin's conjecture reads as follows:
	\begin{conjecture}[\cite{FR2017}]\label{conj:fish}
		There exists a constant $c>0$ such that, for any couple $(x,y)=(x_n,y_n)\in V^2$ and for every $b>0$, w.h.p.,
		\begin{equation}
			\mathbf{Q}\tonde{ W_{cn}(x)\neq W_{cn}(y) } \le n^{-b} \fstop
		\end{equation}
	\end{conjecture}

	The above conjecture can be clearly interpreted  as a meeting problem; however, contrarily to the model we focus on in this paper,  the two random walks in \cref{conj:fish} are \emph{coupled}, i.e.,  they are forced to move following the \emph{same word}. In particular, once such two walks meet, they are doomed to stick together from that moment on. Despite this  difference from our \emph{independent} system,   simulations suggest that the first meeting times of coupled and independent processes share a similar behavior (see \cref{fig}). 
	
	In view of this connection, we conclude that \cref{conj:fish} is false in our setting of independent  walks, as the following consequence of \cref{th:2} shows:
	\begin{corollary}
		For any couple $(x,y)=(x^{(n)},y^{(n)})\in V^2$ of distinct states and any constant $c >0$, w.h.p., 
		\begin{align*}
			\Prob_{(x,y)}\tonde{\tau_{\rm meet}> c n}>\frac{e^{-c}}{2}	\fstop
		\end{align*} 
	\end{corollary}
	\begin{figure}
		\centering
		\includegraphics[width=7cm]{meeting_independent}\qquad\qquad
		\includegraphics[width=7cm]{meeting_coupled}\\ \vspace{0.2cm}
		\includegraphics[width=7cm]{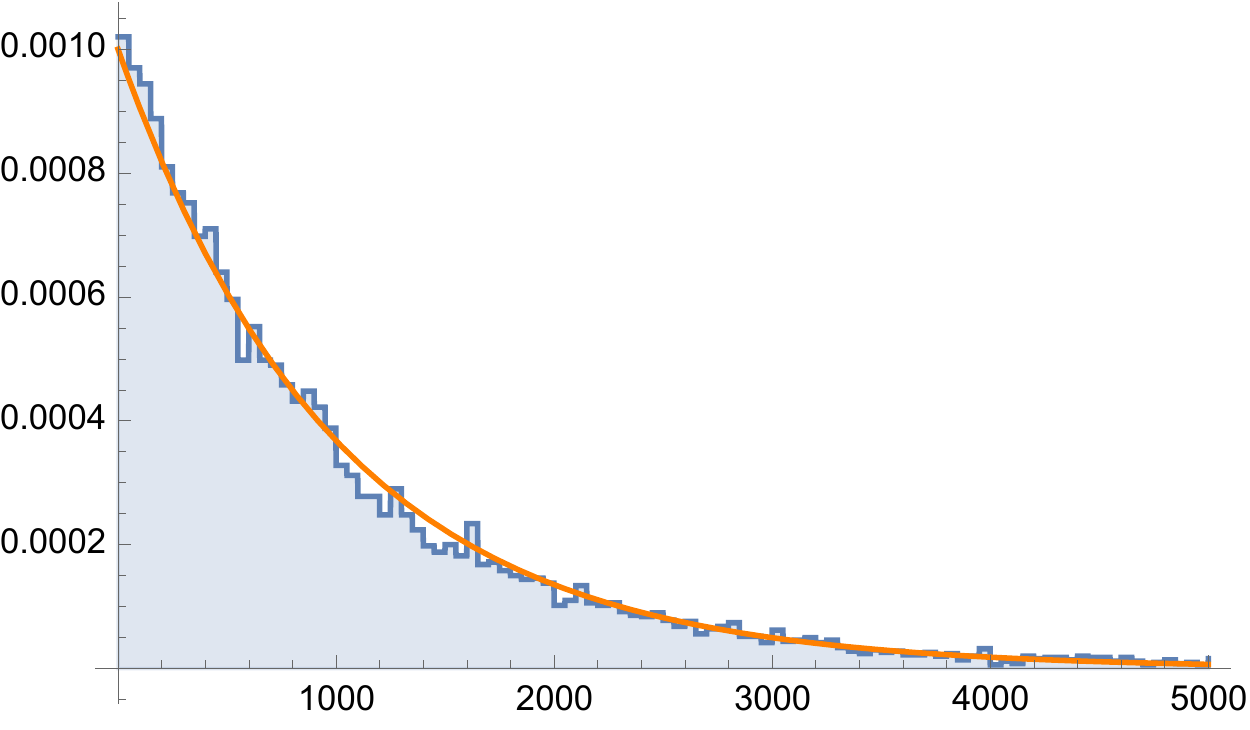}\qquad\qquad
		\includegraphics[width=7cm]{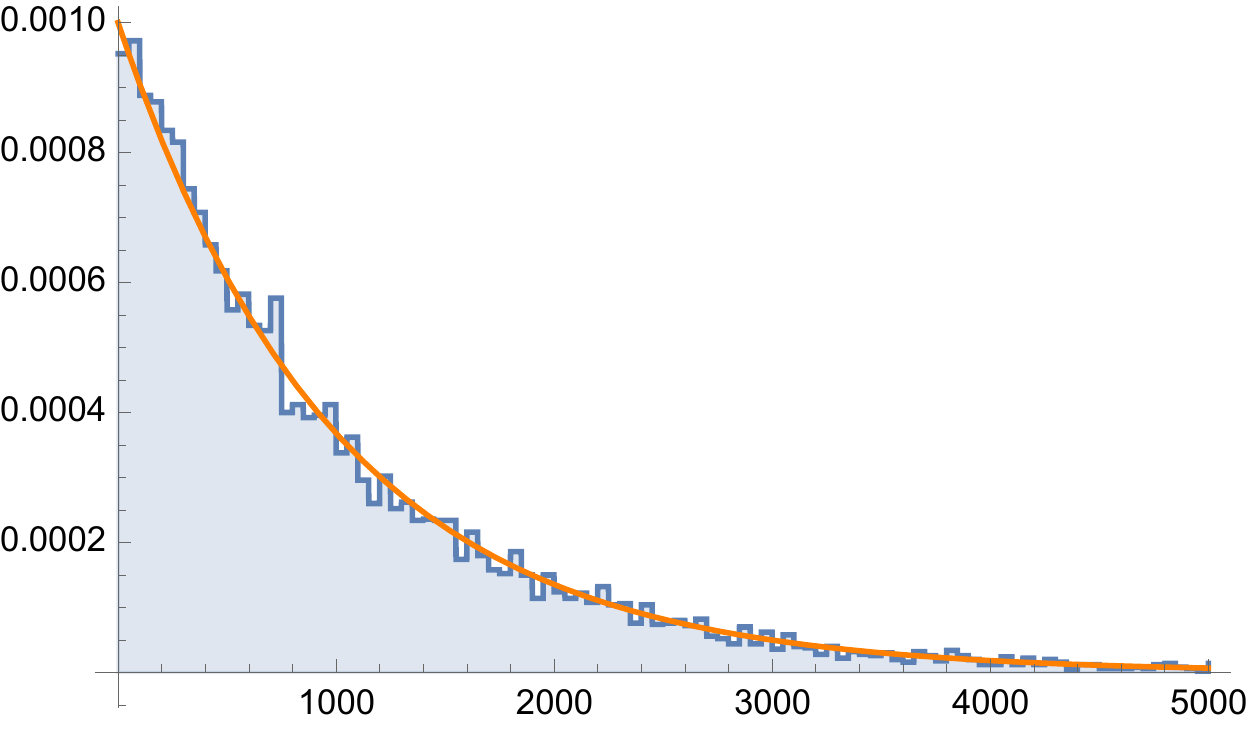}
		\caption{In orange, the PDF of an exponential distribution of mean $n$. In  blue, the empirical PDF of the meeting time of two \emph{independent} (left) and \emph{coupled} (right) random walks starting from two states uniformly at random. The simulations are performed by sampling $10^4$ random DFAs with size $n=1000$. \textcolor{black}{We used $r=2$ for the top row, and $r=20$ for the bottom one.}}\label{fig}
	\end{figure}
	
	\subsubsection{Synchronization of a DFA, \v{C}ern{\`y}'s conjecture, and coalescence}
	Beyond learning theory, DFAs are known to be the object of a long-standing open problem due to \v{C}ern{\`y} \cite{cerny1964}. The so-called \emph{\v{C}ern{\`y}'s conjecture} is related to the notion of \emph{synchronization} of a DFA. A given DFA is  \emph{synchronizable} if there exists a word $w$ such that $w(x)=w(y)$ for every $x,y\in V$; such a word is said to be a synchronizing word for the DFA. Clearly, if a DFA is synchronizable, then there exist  arbitrarily many synchronizing words. The conjecture amounts to the claim that, if a DFA is synchronizable, then the length of the shortest synchronizing word is at most $(n-1)^2$. 
	In that same work \cite{cerny1964}, the author  constructs an example of a DFA having a word  of length exactly $(n-1)^2$ as the shortest synchronizing word. Therefore, if the conjecture were true, then  $(n-1)^2$ would be a sharp bound. 
	Relaxing a bit the problem, one strategy is to look for a high-probability result which ensure the existence of short synchronizing words when the DFA is sampled at random. Along these lines,  Nicaud \cite{nicaud_1,nicaud_2} recently showed that, when the DFA is taken uniformly at random, then there exists a synchronizing word of length $O(n\log^3(n))$ with high probability. More precisely, letting $\tau_{\rm sync}$ denote the smallest $t\in \N$ for which the random word $W_t$ is synchronizing, and  using the notation introduced in \cref{conj:fish}:
	\begin{theorem}[\cite{nicaud_2}]
		W.h.p., there exists a constant $c>0$ such that
		\begin{equation}
			\mathbf{Q}\tonde{ \tau_{\rm sync}\le cn \log^3(n) } \ge r^{-cn\log^3(n)} \fstop
		\end{equation}
	\end{theorem}
	Roughly speaking, this result implies that  \v{C}ern{\`y}'s conjecture holds for most large automata, and that the upper bound $(n-1)^2$ is far from being tight for a typical DFA. Nonetheless, Nicaud's result does not provide an answer to the question \emph{``how rare are such short synchronizing words?''}. More precisely, taking a random word $W_t$ of length $t>0$, and letting $p_t$ be the probability that $W_t$ is synchronizing for a quenched realization of the DFA, what is the behavior of the random sequence $(p_t)_{t\ge 0}$ for large DFAs? 
	
	As for the meeting problem described in \cref{sec:learning}, this synchronization problem may be approximated by means  of a system of \emph{coalescing random walks}, which we now describe. Let $n$   walks start from all distinct vertices,  let them evolve synchronously but \emph{independently} (i.e., each following an independent word), and when two or more particles meet, they merge together  and evolve as a single walk (i.e., they follow the same word only after their meeting). We let $\Prob_{\rm coal}$ denote the law of this Markov chain, and define the \emph{coalescing time} $\tau_{\rm coal}$  as the first time in which only one of the $n$ walks is left. By \cref{th:1} and a union bound, it is immediate to check that
	\begin{align}\label{eq:coal-nlogn}
		\Prob_{\rm coal}\tonde{\tau_{\rm coal}>\(1+\eps\)n \log n} \overset{\P}\longrightarrow 0\comma\qquad \eps>0\fstop
	\end{align}
	Actually, 	since the single random walk on a random DFA satisfies w.h.p.\ the mean field conditions in \cite{oliveira_mean2013},  Theorem 1.2 therein and \cref{pr:fvtltilde} below (that is, essentially the claim in \cref{th:2}, but with the two walks starting independently from stationarity) prescribe\footnote{Note that the results in \cite{oliveira_mean2013} are stated for continuous-time walks.} the limit distribution of $\tau_{\rm coal}$: letting $Z_2, Z_3, \ldots, Z_i,\ldots$ be jointly independent random variables such that $Z_i \sim {\rm Exp}(\binom{i}{2})$, 
	\begin{align}\label{eq:mean-field-behavior}
		d_W\tonde{\frac{\tau_{\rm coal}}{n},\sum_{i=2}^\infty Z_i}\overset{\P}\longrightarrow 0\comma\qquad  
	\end{align}
	where $d_W\tonde{\emparg,\emparg}$ denotes the usual $L^1$-Wasserstein distance. In particular,   \cref{eq:mean-field-behavior} implies
	\begin{align}
		\frac{\mathbf{E}_{\rm coal}\quadre{\tau_{\rm coal}}}{n}\overset{\P}\longrightarrow 2\comma
	\end{align}
	which, by Markov inequality, yields the following  strengthening of \cref{eq:coal-nlogn}: for every $\varepsilon>0$, there exists $c=c_\varepsilon>0$ such that, w.h.p.,
	\begin{align}\label{eq:coal-n}
		\Prob_{\rm coal}\tonde{\tau_{\rm coal}>c n } < \varepsilon\fstop
	\end{align}
	Also in this case, simulations suggest that the two models (synchronization \emph{vs.}\ coalescence)  roughly share the same behavior (see \cref{fig2}). For this reason, it is natural to believe to the following:
	\begin{conjecture}\label{conj:synch}
		Using the notation introduced in \cref{conj:fish},
		\begin{equation}
			\frac{\mathbf{E}_{\mathbf{Q}}\[ \tau_{\rm sync} \]}{n} \overset{\P}{\longrightarrow}2\fstop
		\end{equation}
		Therefore, for every $\varepsilon>0$, there exists $c=c_\varepsilon>0$ such that, w.h.p.,
		\begin{equation}
			\mathbf{Q}\( \tau_{\rm sync}> c n \) \le \varepsilon \fstop
		\end{equation}
	\end{conjecture} 
	Notice that if the latter conjecture held, then it would also provide a sharpening of the results in \cite{nicaud_2}, by proving that there exist synchronizing words of length $O(n)$, and actually most words of length $\omega(n)$ are synchronizing.

	\begin{figure}
		\centering
		\includegraphics[width=7cm]{coalescing_independent}\qquad\qquad
		\includegraphics[width=7cm]{coalescing_coupled}\\ \vspace{0.2cm}
		\includegraphics[width=7cm]{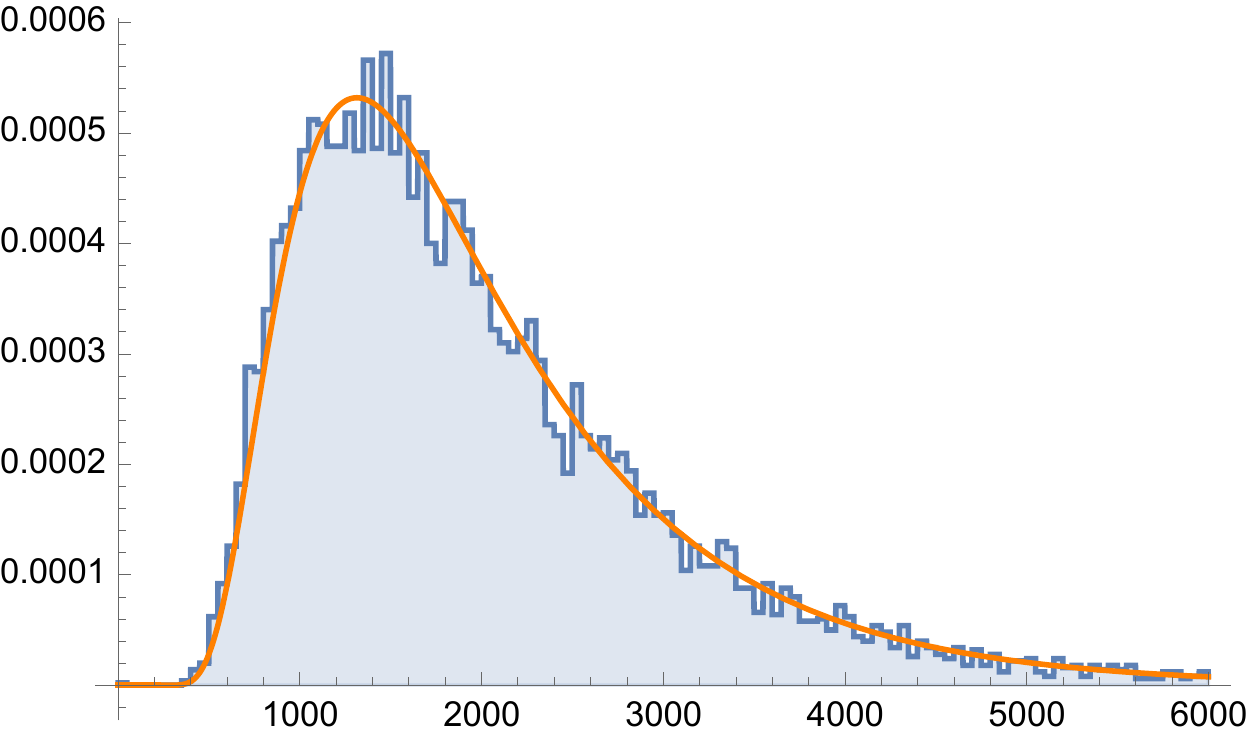}\qquad\qquad
		\includegraphics[width=7cm]{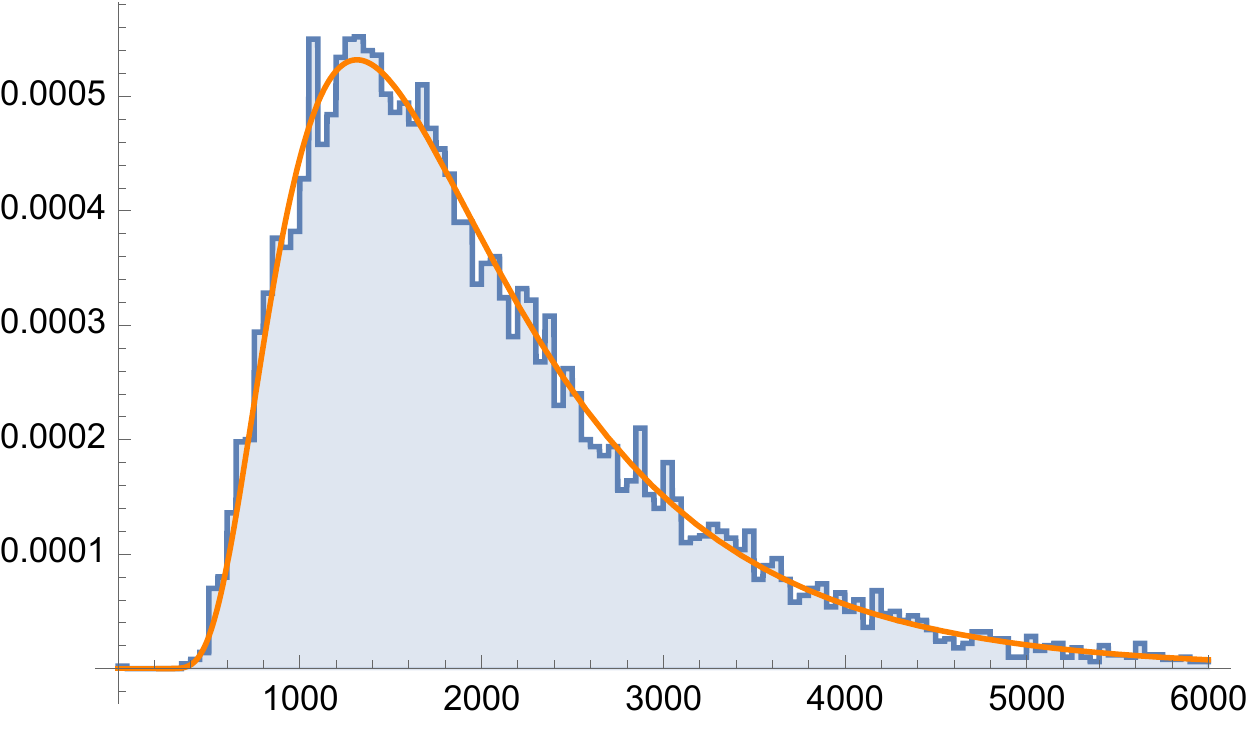}
		\caption{In orange, the PDF of the distribution of the random variable $n \sum_{i=2}^\infty Z_i$, where $Z_i$ is given as in \cref{eq:mean-field-behavior}. In  blue, the empirical PDF of the coalescence time $\tau_{\rm coal}$ (left) and of the synchronization time $\tau_{\rm sync}$  (right). The simulations are performed by sampling $10^4$ random DFAs with size $n=1000$. \textcolor{black}{We used $r=2$ for the first row and $r=20$ for the bottom one.}}
		\label{fig2}
	\end{figure}

	\begin{remark} In the context of random DFA, the condition in \cref{eq:V-C-f} that the $f$'s are one-to-one is often not required (see, e.g., \cite{FR2017,nicaud_2,addario-berry_diameter2020}). (This  condition translates into the constraint  that a random DFA does not display multiple edges with the same origin-destination pair.)
		We impose this condition for the mere scope of importing without  changes all the results in \cite{bordenave_cutoff2019,caputo_quattropani_2021_SPA}, which are based on this assumption.
		
		Nonetheless, it is immediate to check that, even when this constraint is neglected, the number of such multiple edges stays bounded with high probability.
		Given this, it should not be too hard to extend the results therein to the unconstrained setting. Nevertheless, this attempt is out of the scope of the present paper.
	\end{remark}

	\section{Auxiliary chain and First Visit Time Lemma}\label{sec:preliminaries}
	As in other related works (e.g., \cite{cooper_frieze_radzik_2009, oliveira_mean2013}), our strategy of proof is based on interpreting the meeting time  for two walks  as the hitting time of the diagonal
	\begin{align}
		\Delta\eqdef \{(x,x): x\in V\}
	\end{align}
	for the \emph{product chain} $\mathbf{X}^{(2)}_t=(X^{(1)}_t,X^{(2)}_t)$. Clearly, such a  hitting time is independent on transition probabilities \emph{from} the diagonal, therefore in this analysis the product chain may be replaced by any other chain behaving as $\mathbf{X}^{(2)}_t$ until the first hitting of $\Delta$. 
	
	In what follows, we adopt this idea, introducing an effective auxiliary process (\cref{sec:auxiliary-chain}) for which the hypothesis of the First Visit Time Lemma (\cref{fvtl} in \cref{sec:fvtl}) are shown to hold (\cref{lemma:hpfvtl} in \cref{sec:hpfvtl}).
	
	\subsection{Auxiliary chain $\Xi_t$}\label{sec:auxiliary-chain}
	Fix a realization of the random DFA $G$, and fix a stationary measure $\pi$ for the associated chain.  In this setting,  we introduce an \emph{auxiliary chain}	 $(\Xi_t)_{t\in \N}$ on the state space
	\begin{align}\label{def:Vtilde}
		\tilde V\eqdef V^2_{\not=}\sqcup \{\Delta\}\eqdef \left\{(x,x')\in V^2: x\neq x'\right\} \sqcup \{\Delta\}\comma
	\end{align} 
	namely the set $V^2$ in which elements in $\Delta$ are identified, \textcolor{black}{and now $\Delta$ is considered as a state for this new chain\footnote{We emphasize that, when working with $V^2$, $\Delta$ is a subset of states; when working with $\tilde V$, $\Delta$ is considered as a  state.}.}
	\textcolor{black}{In words, such a Markov chain  has the same behavior as that of two independent walks when the two walks are off the diagonal. When the two walks reach the diagonal $\Delta$, then they move independently out of  the same vertex $z \in V$ sampled with probability proportional to $\pi(z)^2$.} 
	More precisely, the law of such a chain (given the underlying DFA $G$), which will be referred to as $(\tilde \Prob_\xi)_{\xi\in \tilde V}$, is the one induced by the transition matrix   $\tilde P$ given by (here, $(x,x'), (y,y')\in V^2_{\not=}$)
	\begin{align*}
		\tilde P(\xi,\zeta)\eqdef \begin{cases}
			P(x,y)P(x',y') &\text{if}\ \xi=(x,x')\comma \zeta	=(y,y')\\
			\sum_{z\in V} P(x,z)P(x',z) &\text{if}\ \xi=(x,x')\comma \zeta = \Delta\\
			\sum_{z\in V}\frac{\pi(z)^2}{\sum_{w\in V}\pi(w)^2} P(z,y)P(z,y') &\text{if}\ \xi= \Delta\comma \zeta=(y,y')\\
			\frac{1}{r} &\text{if}\ \xi, \zeta\in \Delta\fstop
		\end{cases}
	\end{align*}
	As already observed in \cite[\S2.3]{manzo_quattropani_scoppola_2021}, whenever the chain $P$ admits $\pi$ as its unique stationary measure, then 
	\begin{align}\label{def:pitilde}
		\tilde\pi(\xi)\eqdef \begin{cases}
			\pi(x)\pi(x') &\text{if}\ \xi=(x,x')\\
			\sum_{z\in V}\pi(z)^2 &\text{if}\ \xi=\Delta\fstop
		\end{cases}	
	\end{align} is the unique stationary  measure  for $\tilde P$.

	\subsection{First Visit Time Lemma}\label{sec:fvtl}
	Given a growing sequence of Markov chains, the so-called \emph{First Visit Time Lemma} (FVTL) \cite{cooper_frieze_2004} (see also \cite{manzo_quattropani_scoppola_2021}) is a powerful tool for the asymptotic analysis of  hitting times when starting from stationarity.  Originally motivated by the study of cover times of random walks on random graphs, Cooper and Frieze developed this criterion and successfully applied it to several problems (see, e.g., \cite{cooper_frieze_2004,cooper_frieze_2005,cooper_frieze_2007,cooper_frieze_2008}). More recently,  the authors in \cite{manzo_quattropani_scoppola_2021} provided a new proof of such a lemma, linking this result to the theory of quasi-stationary distributions and metastability for Markov chains, \textcolor{black}{in the spirit of previous works from the '80, see, e.g., \cite{aldous82}}. 
	
	Before presenting a detailed version of the theorem, we briefly explain in words its content. To this purpose, consider a (discrete-time) ergodic Markov chain  on a finite set $[N]$, with transition matrix $Q$, and with stationary measure $\mu$; further, consider the corresponding mixing times, i.e.,   
	\begin{color}{black}
	\begin{equation}\label{eq:def-tmix} 
		t_{\rm mix}=t_{\rm mix}(Q)\eqdef\inf\left\{ t\in \N\:\bigg\rvert\:\max_{z\in[N]} \left\|Q^t(z,\emparg)-\mu\right\|_{\rm TV}\le \frac{1}{2e} \right\}\comma
	\end{equation}
where $\left\|\nu_1-\nu_2\right\|_{\rm TV}$ denotes the total-variation distance between two probability measures $\nu_1$ and $\nu_2$ defined on the same space.
Roughly speaking,	the  FVTL asserts   that for a growing (i.e., $N\to \infty$) sequence of  Markov chains in which the mixing time is sufficiently small compared to the stationary measure of a target state, then the hitting times of  such a target state is  geometrically distributed  when starting from stationarity.
\end{color}
	\begin{color}{black}
		\begin{theorem}[FVTL]\label{fvtl}
			Consider a sequence of ergodic Markov chains with state spaces $[N]$, transition matrices $Q=Q_N$, and  unique stationary measures $\mu=\mu_N$. Further, consider a sequence of target states $\partial=\partial_N \in {\rm supp}(\mu)\subseteq[N]$
			and assume that
			\begin{equation}\label{eq:small-mixing}
				\mu(\partial)\,t_{\rm mix}\, \log\(\frac{1}{\min_{z\in {\rm supp}(\mu)}\mu(z))}\)\underset{N\to \infty}\longrightarrow 0\fstop
			\end{equation}
			Then, there exists some $\lambda=\lambda_N\in(0,1)$ such that 
			\begin{equation}\label{eq:geometric0}
				\sup_{t\ge 0}\left|\frac{\Pr(\tau_\partial>t \:|\: X_0\sim \mu)}{(1-\lambda)^t}-1 \right|\underset{N\to \infty}\longrightarrow 0\fstop
			\end{equation}
			Moreover, for any   sequence $T=T_N$ satisfying
			\begin{equation}\label{eq:hp}
				T\ge 2t_{\rm mix}\, \log\(\frac{1}{\min_{z\in {\rm supp}(\mu)}\mu(z)}\)\qquad\text{and}\qquad	\mu(\partial)\,  T \underset{N\to \infty}\longrightarrow 0\comma
			\end{equation}
			we have
			\begin{align}\label{eq:lambda}
				\frac{\lambda}{\mu(\partial)/R}\underset{N\to \infty}\longrightarrow 1\comma\qquad \text{with}\  R=R_{N,T}\eqdef\sum_{t=0}^T Q^t(\partial,\partial)\fstop
			\end{align}
		
		\end{theorem}

		Henceforth, the FVTL  not only asserts  that  the mixing condition in \cref{eq:small-mixing} guarantees the asymptotic geometric distribution of the hitting time of the target (cf.\ \cref{eq:geometric0}), but also identify the asymptotic behavior of the parameter of the geometric distribution (cf.\ \cref{eq:lambda}). Indeed, as \cref{eq:lambda} shows,  $\lambda$ is asymptotically prescribed by:
		\begin{itemize}
			\item $\mu(\partial)$, the stationary value of the target;
			\item $R$, the mean number of returns to the target within time $T$.
		\end{itemize}   
	\end{color}
		
\begin{color}{black}		Finally, we remark that this version of the FVTL is a slightly more convenient rewriting of the one presented in \cite[Theorem 2.2]{manzo_quattropani_scoppola_2021}. The main difference is that here we do not assume the sub-Markovian chain $[Q]_\partial$ (in which the row and column associated to the target state $\partial$ have been erased) to be irreducible. This condition is not crucial, as already pointed out, e.g., in \cite[Remark 3.8]{aldous82}. For the sake of completeness, we report a complete and self-contained proof of \cref{fvtl} in \cref{apx:FVTL}.
		\end{color}

	\subsection{Auxiliary chain and FVTL}\label{sec:hpfvtl}
	We now apply the FVTL to the auxiliary chain $\Xi_t$ introduced above. In this context, $N=n(n-1)+1$,  $Q=\tilde P$,  $\mu=\tilde \pi$, and $\partial=\Delta$. In particular, recall that $[\tilde P]_\Delta$ denotes the sub-Markovian transition matrix obtained by $\tilde P$ by removing the state $\Delta$.  Therefore, 		in order to verify the assumptions of \cref{fvtl}, it suffices to show the validity of the following lemma.

	\begin{proposition}\label{lemma:hpfvtl}
		Let $G$ be a random DFA, and consider the process $\Xi_t$ defined in \cref{sec:auxiliary-chain}. Letting  $T\eqdef\lceil\log^5(n)\rceil$, \textcolor{black}{$S\eqdef\lceil\log^3(n)\rceil$}, and  $\varepsilon\in (0,1)$, we	  consider the following events:
		\begin{align}
			\label{hp:pimin}\cA_1&\eqdef\left\{\min_{\xi\in{\rm supp}(\tilde \pi)\subseteq\tilde V}\tilde\pi(\xi)\ge n^{-3.6}\right\}\comma\\
			\label{hp:pimax}\cA_2&\eqdef\left\{\max_{\xi\in\tilde V}\tilde\pi(\xi)\le \frac{\log^8(n)}{n}\right\}\comma\\	\label{hp:pidelta}\cA_3&\eqdef\left\{\abs{n\,\tilde\pi(\Delta)-\frac{r}{r-1}}<\varepsilon\right\}\comma\\
			\label{hp:mixing}\cA_4&\textcolor{black}{\eqdef\left\{\max_{\xi\in \tilde V}\|\tilde P^S(\xi,\cdot)-\tilde\pi\|_{\rm TV}< \varepsilon\right\}}\comma\\
			\label{hp:R}\cA_5&\eqdef\left\{\abs{\tonde{\sum_{t=0}^T\tilde P^t(\Delta,\Delta)}-\frac{r}{r-1}}<\eps\right\}
			\fstop
		\end{align}
		Then, for every $\varepsilon >0$, $\cap_{i=1}^5\,\cA_i$ occurs w.h.p..
	\end{proposition}
	\cref{fvtl,lemma:hpfvtl}, and the fact that  $P\otimes P$ and $\tilde P$ coincide out of $\Delta$, immediately yield the following result:
	\begin{proposition}\label{pr:fvtltilde}
		Let $G$ be a random DFA and consider two independent walks on $G$. Then, there exists a sequence of random variables $\Lambda=\Lambda_n\in(0,1)$ such that
		\begin{equation}
			n\,\Lambda\overset{\P}\longrightarrow 1\comma\qquad	\sup_{t\ge 0}\left|\frac{\Prob_{\pi\otimes\pi}(\tau_{\rm meet}>t)}{(1-\Lambda)^t}-1 \right|\overset{\P}\longrightarrow 0 \fstop
		\end{equation}
	\end{proposition}

	\cref{sec:meeting-stationarity} is devoted to the proof of \cref{lemma:hpfvtl}. In \cref{sec:theorems_proofs} we use \cref{pr:fvtltilde} to deduce \cref{th:1,th:2}.

	\begin{color}{black}
		\begin{remark}\label{rmk:indep-r-2}
			As already pointed out in right below \cref{coro:mean-meeting}, the asymptotic distribution of the meeting time does not depend on $r$, the out-degree. Indeed, while the intuition that it should depend on  $r$ seems plausible, actually --- as the First Visit Time Lemma rigorously prescribes --- the mean of the meeting  time asymptotically depends on the  ratio between the following two quantities: \textit{the stationary measure of the diagonal} over \textit{the expected sojourn time on the diagonal} (within the mixing time). Since both such quantities are asymptotically equal (up to normalization) to $r/(r-1)$ (see \cref{hp:pidelta,hp:R}), the dependence on $r$ cancels out, w.h.p.,  in the  asymptotic distribution of the meeting time. 
		\end{remark}
	\end{color}

	\section{Meeting time starting from stationarity. Proof of \cref{lemma:hpfvtl}}\label{sec:meeting-stationarity}
	Throughout the rest of the paper, for notational convenience, we omit writing the integer part $\ceil{\emparg}$ of all time variables.
	
	In order to prove \cref{lemma:hpfvtl}, we start by recalling some known results on the behavior of a single random walk and its stationary measure on the random DFA $G$.
	\begin{theorem}[\cite{bordenave_cutoff2019, addario-berry_diameter2020,  caputo_quattropani_2021_SPA}]\label{th:single_rw}Let $G$ be a random DFA.
		\begin{itemize}
			\item \textbf{Uniqueness of the stationary measure} (\cite[Theorem 1]{bordenave_cutoff2019}): w.h.p., 
			\begin{equation}\label{eq:uniquepi}
				\exists!\, \pi: \pi P=\pi\fstop
			\end{equation}
			\item \textbf{Mixing with cutoff} (\cite[Theorem 1]{bordenave_cutoff2019}): for $\alpha>0$ and $t_\alpha\eqdef \alpha\log(n)$,	w.h.p.,
			\begin{align}\label{eq:cutoff}
				\max_{x\in V}\left| \| P^{t_\alpha}(x,\cdot)- \pi \|_{\rm TV}-\ind_{\(-\infty,\frac{1}{\log(r)}\)}(\alpha)\right|\overset{\P}{\longrightarrow}
				0\comma\qquad \alpha\neq \frac{1}{\log(r)}\fstop
			\end{align}
			\item \textbf{Minimum of $\pi$} (\cite[Theorem 35]{addario-berry_diameter2020}): w.h.p.,
			\begin{align}\label{eq:min-pi}
				\min_{x\in {\rm supp}(\pi)}\pi(x)\ge \frac{1}{n^{1.8}}\fstop
			\end{align}
			\item \textbf{Maximum of $\pi$} (\cite[Lemma 4.2]{caputo_quattropani_2021_SPA}): w.h.p.,
			\begin{align}\label{eq:max-pi}
				\max_{x\in V}\pi(x)\le \frac{\log^8(n)}{n}\fstop
			\end{align}
		\end{itemize}
	\end{theorem}
	
	In this rest of this section, we focus on the auxiliary chain $\Xi_t$ introduced in \cref{sec:preliminaries}.
	Let us observe that, since \cref{eq:uniquepi} occurs w.h.p., when proving \cref{lemma:hpfvtl}, we will implicitly assume that the random DFA $G$ gives rise to an ergodic chain $(P,\pi)$; by the discussion at the end of \cref{sec:auxiliary-chain}, the auxiliary chain $\Xi_t$ has a unique stationary measure $\tilde \pi$ as given in \cref{def:pitilde}.

	\begin{color}{black}
		\subsection{Organization of the proof of \cref{lemma:hpfvtl}}\label{sec:organization}
		The rest of the section is divided into three parts. 
		
		In \cref{suse:pi}, we control the probability of the events $\cA_1$, $\cA_2$ and $\cA_3$, i.e., we bound extremal entries of $\tilde{\pi}$ and provide the first order asymptotics of  $\tilde{\pi}(\Delta)$. While the former control easily follows from \cref{th:single_rw} and is the content of \cref{lemma:pi-diag,lemma:max-pi}, the  latter requires a deeper analysis, which we carry out in \cref{lemma:A2}.
		
		In \cref{suse:mixing}, we analyze the mixing time of the auxiliary chain, showing in \cref{prop:mix-diag} that,  w.h.p., $\cA_4$ holds.  The proof of this result relies on the mixing result in \cref{th:single_rw} for a single walk and on a coupling of the the auxiliary chain with two independent walks. This part is divided into three main lemmas, \cref{lemma:mu+,lemma:mu-xy-new,prop:next-meeting}, which essentially show that, once the auxiliary chain exits the state $\Delta$, it can be coupled with the product chain for a polylogarithmic number of steps at a small TV-cost. 
		
		Finally, exploiting the tools developed in \cref{suse:mixing}, in \cref{suse:returns} we focus on the number of returns to $\Delta$ for the auxiliary chain, ensuring that, w.h.p.,  $\cA_5$ holds.
	\end{color}
	\subsection{Estimating $\tilde \pi$}\label{suse:pi} Recall the events $\cA_1$, $\cA_2$ and $\cA_3$ in \cref{hp:pimin,hp:pidelta,hp:pimax}. In the following three lemmas, we respectively show that $\lim_{n\to \infty}\P(\cA_i)=1$ for $i=1,2,3$.	
	\begin{lemma}[Minimum of $\tilde \pi$]\label{lemma:pi-diag}
		$\lim_{n\to \infty}\P(\cA_1)=1$.
	\end{lemma}
	\begin{proof}
		By Cauchy-Schwarz inequality,
		$\tilde\pi(\Delta)\ge n^{-1}$, while
		\cref{def:pitilde,eq:min-pi} yield		
		\begin{align*}\min_{\xi\neq\Delta}\tilde\pi(\xi)\ge \frac{1}{n^{3.6}}\fstop
		\end{align*}
		This concludes the proof of the lemma.
	\end{proof}
	\begin{lemma}[Maximum of $\tilde \pi$]
		\label{lemma:max-pi}
		$\lim_{n\to \infty}\P(\cA_2)=1$.	
	\end{lemma}
	\begin{proof}
		By the definition of $\tilde\pi$ in \cref{def:pitilde}, H\"older inequality yields 
		\begin{align*}
			\max_{\xi\in \tilde V}\tilde\pi(\xi)\le \max\left\{\max_{(x,y)\in V^2_{\not=}} \pi(x)\pi(y)\comma \tilde\pi(\Delta)\right\}\le \max_{x\in V} \pi(x)\fstop
		\end{align*}
		\cref{eq:max-pi} concludes the proof of the lemma.
	\end{proof}
	\begin{lemma}[Value of $\tilde\pi(\Delta)$]\label{lemma:A2}
		$\lim_{n\to \infty}\P(\cA_3)=1$, for every $\eps>0$.
	\end{lemma}
	\begin{proof}
		Recall the definition of $\tilde \pi$ from \cref{def:pitilde}, and fix $t\eqdef \log^3(n)$. Instead of proving the desired claim directly, we first show that, letting
		\begin{align}\label{eq:Y}
			Y\eqdef \frac{1}{n^2}\sum_{y,z\in V}\sum_{x\in V}P^t(y,x)P^t(z,x)\comma
		\end{align}
		the following two claims hold:
		\begin{align}\label{eq:first-moment}
			\E\quadre{Y} = \frac{1}{n}	 \frac{r}{r-1}+o\tonde{\frac{1}{n}}\comma
		\end{align}
		and
		\begin{align}\label{eq:second-moment}
			\E[Y^2]\le \E\quadre{Y}^2+ o\tonde{\frac{1}{n^2}}\fstop
		\end{align}
		\cref{eq:first-moment,eq:second-moment} conclude the proof of the lemma. Indeed, by the triangle and Chebyshev inequalities, 
		\begin{align*}
			&\P\tonde{\abs{n\tilde \pi(\Delta)-\frac{r}{r-1}}>\eps}
			\\
			&\qquad\le 	\P\tonde{\abs{n\tilde \pi(\Delta)-nY}>\frac{\varepsilon}{2}}+	n^2\frac{\E[Y^2]-\E\quadre{Y}^2}{(\varepsilon/4)^2} + \car_{(\frac{\varepsilon}{4},\infty)}\tonde{\abs{n\E\quadre{Y}-\frac{r}{r-1}}}\fstop
		\end{align*}
		While the second and third terms on the right-hand side above vanish as $n\to \infty$ by \cref{eq:first-moment,eq:second-moment}, the first term vanishes  by   the fact that  $t$ is order $\log^2(n)$ times the mixing time (see \cref{eq:cutoff}):
		\begin{align*}
			\lim_{n\to \infty}\P\tonde{	\max_{x,y\in V}\abs{P^t(x,y)-\pi(y)}\le n^{-2}}=1\fstop
		\end{align*}
		
		We are left to show the validity of \cref{eq:first-moment,eq:second-moment}. As a general strategy, we employ a system of four \emph{annealed random walks} (see \cite[Section 2.2]{bordenave_cutoff2019}) running  for a time $t=\log^3(n)$.	Roughly speaking, starting from an empty environment, we construct the whole trajectories of these walks \emph{one at the time}, and concurrently  construct the environment that these walks explore. More precisely, let
		\begin{align*}
			\tonde{\tonde{Z^{(1)}_s,Z^{(2)}_s,Z^{(3)}_s,Z^{(4)}_s}}_{s=0}^t \in (V^4)^{t+1}\comma
		\end{align*}
		be the non-Markovian process with law $\P^{\rm an}$ constructed as follows:
		\begin{enumerate}[(i)]
			\item Initially,  set the environment, say $\sigma^{(1)}$, to consist of an \textquotedblleft empty graph\textquotedblright, i.e.,  $\sigma^{(1)}_0\eqdef \emp$.
			\item Select a uniformly random vertex $y\in V$, and consider a walk $Z^{(1)}$ starting at $y$, i.e., $Z^{(1)}_0\eqdef y$.
			\item At every step $s\in \{0,\ldots, t\}$, 	 given the
			current environment $\sigma^{(1)}_s$ and position of the walk $Z^{(1)}_s$, the walk picks a uniformly random color $c\in\cC$ and looks at the associated out-going edge from $Z^{(1)}_s$:
			\begin{itemize}
				\item If the $c$-tail of the vertex $Z^{(1)}_s$ is unmatched, select a uniformly random destination among all vertices in $V$ which have no directed edge from $Z^{(1)}_s$, yet. Then, call $\sigma^{(1)}_{s+1}$ the new environment obtained from $\sigma^{(1)}_s$ by adding this new edge, and move the walk to this vertex.
				\item If the $c$-tail of the vertex $Z^{(1)}_s$ is already matched, i.e., the $c$-colored directed out-going edge from $Z^{(1)}_s$ already belongs to the environment $\sigma^{(1)}_s$, then simply set $\sigma^{(1)}_{s+1}\eqdef \sigma^{(1)}_s$, and   move the walk to the end-point of the $c$-tail attached to $Z^{(1)}_s$.
			\end{itemize}
			\item Once the first walk $Z^{(1)}$ has completed its trajectory of length $t$, perform the same procedure for the second walk $Z^{(2)}$, but this time starting with the environment $\sigma^{(2)}_0\eqdef \sigma^{(1)}_t$, i.e., the environment already revealed by the trajectory of $Z^{(1)}$. Similarly for $Z^{(3)}$ and $Z^{(4)}$, respectively with starting environments $\sigma^{(3)}_0\eqdef \sigma^{(2)}_t$ and $\sigma^{(4)}_0\eqdef \sigma^{(3)}_t$. 
		\end{enumerate} 
		These annealed walks provide us with an alternative expression for $\E[Y]$ and $\E[Y^2]$: recalling $Y$ in \cref{eq:Y},
		\begin{align}\label{eq:first-moment-an}
			\E[Y]=\P^{\rm an}\ttonde{Z^{(1)}_t=Z^{(2)}_t}\comma
		\end{align}
		and
		\begin{align}\label{eq:second-moment-an}
			\E[Y^2]=\P^{\rm an}\ttonde{Z^{(1)}_{t}=Z^{(2)}_t\comma Z^{(3)}_{t}=Z^{(4)}_t}\fstop
		\end{align}
		
		We start with the proof of \cref{eq:first-moment} using \cref{eq:first-moment-an}.
		To this purpose, letting	
		\begin{align}\label{eq:N}
			\cN\eqdef \bigsqcup_{x\in V}\cN_x\eqdef \bigsqcup_{x\in V}\{Z^{(1)}_t=Z^{(2)}_t=x \} 
			\comma
		\end{align}
		we have
		\begin{align}\label{eq:first-moment-an-sum}
			\P^{\rm an}\ttonde{Z^{(1)}_{t}=Z^{(2)}_t}=\sum_{x\in V}\P^{\rm an}(\cN_x)\comma
		\end{align}
		and, by symmetry,  all the summands in the last display are equal.
		Therefore, fix any arrival point $x\in V$ for the two walks, and define the events
		\begin{align*}
			\cN^{(i)}_x\eqdef\{Z^{(i)}_t=x \}\comma\qquad i=1,2\fstop
		\end{align*}
		We now show
		\begin{align}\label{eq:nx}
			\P^{\rm an}\ttonde{\cN_x}=\tonde{1+o(1)}\frac{1}{n^2}\frac{r}{r-1}\comma
		\end{align}
		from which \cref{eq:first-moment} follows (combine \cref{eq:nx} with \cref{eq:first-moment-an,eq:first-moment-an-sum}).
		The proof of \cref{eq:nx} goes through the following steps:
		\begin{itemize}
			\item 
			Consider  the event for $Z^{(1)}$ of arriving at $x\in V$  \emph{performing a loop}, i.e.,
			\begin{align*}\cN^{(1),{\rm bad}}_x\eqdef \cN^{(1)}_x\cap \cL^{(1)}\eqdef \cN^{(1)}_x\cap\{Z^{(1)}_s=Z^{(1)}_{s'}\ \text{for some}\ s< s' \le t \}\comma
			\end{align*}
			and let
			\begin{align*}\cN_x^{(1),{\rm good}}\eqdef \cN^{(1)}_x\setminus \cN^{(1),{\rm bad}}_x
			\end{align*}
			denote the event that $x$ was hit at time $t$ \emph{without loops}.
			In order to estimate $\P^{\rm an}\ttonde{\cN_x^{(1),{\rm bad}}}$, we further distinguish the case in which $x$ was ever hit before time $t$; thus, letting $[Z^{(1)}]\eqdef \{Z_0^{(1)},\ldots,Z_{t-1}^{(1)}\}$ and $\cH^{(1)}_x\eqdef \{x\in [Z^{(1)}]\}$,  
			\begin{align}\label{eq:n1bad}\begin{aligned}
					\P^{\rm an}\ttonde{\cN^{(1),{\rm bad}}_x}	
					&\le \P^{\rm an}\ttonde{\cN^{(1)}_x\mid  (\cH^{(1)}_x)^\complement \cap \cL^{(1)}} \P^{\rm an}\ttonde{\cL^{(1)}}+ \P^{\rm an}\ttonde{\cN^{(1)}_x\cap \cH^{(1)}_x}\\
					&\qquad\le \frac{1}{n}\times \frac{t^2}{n} + 	\frac{t^3}{n^2}\le 	\frac{2t^3}{n^2}\fstop
				\end{aligned}
			\end{align} 
			Indeed,  $\P^{\rm an}\ttonde{\cN^{(1)}_x\mid  (\cH^{(1)}_x)^\complement \cap \cL^{(1)}} \le \frac{1}{n}$ holds because the event requires to connect to vertex $x$ at time $t$; $\P^{\rm an}\tonde{\cL^{(1)}}\le \frac{t^2}{n}$ comes from estimating by a union bound the probability of the event that, within  $t$ steps, the walk ever hits  one of the previously visited vertices, which are at most $t$. Finally, $\P^{\rm an}\tonde{\cN_x^{(1)}\cap \cH_x^{(1)}}$ is estimated by the probability that the walk  visits $x$ for the first time within time $t-1$ (this occurs with probability less than $\frac{t}{n}$), and then visits one of the vertices which have been previously visited (this \textcolor{black}{happens} with probability less than $\frac{t^2}{n}$).
			\item By an analogous argument and \cref{eq:n1bad},  we obtain
			\begin{align}\label{eq:n1bad-n2}\P^{\rm an}\ttonde{\cN_x^{(1),\rm bad}\cap \cN_x^{(2)}}\le \P^{\rm an} \(\cN_x^{(1),\rm bad}\)\times  \frac{t^2}{n}\le \frac{2t^5}{n^3}\fstop
			\end{align}
			\item\label{it:3(2.6)} We now estimate $\P^{\rm an}\ttonde{\cN^{(1),{\rm good}}_x\cap \cN^{(2)}_x}$. Under $\cN_x^{(1),{\rm good}}$, the second walk $Z^{(2)}$ can reach the same $x\in V$ at time $t$ in either one of the following two ways:
			\begin{itemize}
				\item\label{it:3(a)} $Z^{(2)}$ hits the trajectory of the first walk for the first time at time $s \le t$ \textcolor{black}{in the unique vertex that is at} distance $t-s$ from $x$, and then follows the same path: letting $\{Z^{(1)}\}\eqdef \{Z^{(1)}_0,\ldots, Z^{(1)}_t\}$,
				\begin{align*}&\cN^{(2),{\rm good}}_x\eqdef \bigsqcup_{s=0}^t \{Z^{(2)}_{s'}\notin \{Z^{(1)}\}\ \text{for all}\ 0\le s'< s\}\cap\{Z^{(1)}_{s'}= Z^{(2)}_{s'}\ \text{for all}\ s\le s'\le t\}\fstop
				\end{align*}
				Then, since $t=\log^3(n)$,
				\begin{align}\label{eq:n1good-n2good}
					\begin{aligned}
						\P^{\rm an}\ttonde{\cN^{(1),{\rm good}}_x\cap \cN^{(2),{\rm good}}_x}&= 	\P^{\rm an}\ttonde{ \cN^{(2),{\rm good}}_x\mid \cN^{(1),{\rm good}}_x}\tonde{1+o(1)}\frac{1}{n}\\
						&=\tonde{\sum_{s=0}^{t}\tonde{1+O\tonde{\frac{t}{n}}}^{s}\frac{1}{n}\tonde{\frac{1}{r}}^{t-s}} \tonde{1+o(1)}\frac{1}{n}\\
						&=\frac{1}{n^2}\tonde{ \frac{r}{r-1}+o(1)}\fstop
					\end{aligned}
				\end{align}
				\textcolor{black}{where the first asymptotic equality follows from the definition of $\cN^{(1),{\rm good}}_x$ and \cref{eq:n1bad}.}
				\item \label{it:3(b)} $Z^{(2)}$ hits at some time the path of the first walk, exits at least once the path, and eventually re-enters that same path: recalling $\{Z^{(1)}\}\eqdef \{Z^{(1)}_0,\ldots, Z^{(1)}_t\}$,
				\begin{align*}
					\cN^{(2),{\rm bad}}_x\eqdef\{Z_{s_1}^{(2)}, Z_{s_3}^{(2)}\in \{Z^{(1)}\}\comma Z_{s_2}^{(2)}\notin \{Z^{(1)}\}\comma \text{for some}\ 0\le s_1<s_2<s_3\le t  \}\fstop
				\end{align*}
				Note that $\cN^{(2),{\rm bad}}_x\neq \ttonde{\cN^{(2)}_x\setminus \cN^{(2),{\rm good}}_x}$, but
				\begin{align*}\cN^{(1),{\rm good}}_x\cap \ttonde{\cN_x^{(2)}\setminus \cN^{(2),{\rm good}}_x}\subset \cN^{(1),{\rm good}}_x\cap\cN^{(2),{\rm bad}}_x  \fstop
				\end{align*}
				Hence, we obtain 
				\begin{equation}\label{eq:n1good-n2bad}
					\begin{aligned}
						\P^{\rm an}\ttonde{\cN^{(1),{\rm good}}_x\cap \ttonde{\cN_x^{(2)}\setminus \cN^{(2),{\rm good}}_x}}
						&\le \P^{\rm an}(\cN^{(1),{\rm good}}_x\cap\cN^{(2),{\rm bad}}_x)\\
						&\le (1+o(1))\frac{1}{n}\times \frac{t^2}{n}\times \frac{2 t^2}{n}=o\tonde{\frac{1}{n^2}}\fstop
					\end{aligned}
				\end{equation}
			\end{itemize} 
			\item  In conclusion, since
			\begin{align*}
				\P^{\rm an}\ttonde{\cN_x^{(1),{\rm good}} \cap \cN_{x}^{(2),{\rm good}}}\le	\P^{\rm an}\ttonde{\cN_x}&\le \P^{\rm an}\ttonde{\cN_x^{(1),{\rm good}} \cap \cN_{x}^{(2),{\rm good}}}\\
				&\qquad+\P^{\rm an}\ttonde{ \cN_x^{(1),{\rm bad}}\cap \cN_x^{(2)}}\\
				&\qquad+ \P^{\rm an}\ttonde{\cN_x^{(1),{\rm good}}\cap \cN_x^{(2),{\rm bad}}}\comma
			\end{align*}
			the estimates in \cref{eq:n1bad-n2,eq:n1good-n2good,eq:n1good-n2bad} show the validity of \cref{eq:nx}.
		\end{itemize}

		This concludes the proof of \cref{eq:first-moment}; we now prove \cref{eq:second-moment} using \cref{eq:second-moment-an}. In analogy with \cref{eq:N}, define
		\begin{align*}
			\cM\eqdef \bigsqcup_{y\in V}\cM_y\eqdef \bigsqcup_{y\in V}\{Z^{(3)}_t=Z^{(4)}_t=y\}\comma
		\end{align*}
		and note that, by symmetry,
		\begin{align}\label{eq:N-M}
			\P^{\rm an}\ttonde{\cM}=\P^{\rm an}\ttonde{\cN}= \sum_{x\in V}\P^{\rm an}\tonde{\cN_x}= \E\quadre{Y}\fstop
		\end{align}
		Define further the following events:
		\begin{align*}\cM^{\rm bad}\eqdef\cM\cap \{Z^{(i)}_s=Z^{(j)}_{s'} \ \text{for some}\  i\in \{1,2 \},j \in\{3,4\}, s,s'\in\{0,\dots,t \} \}\comma
		\end{align*}
		and  
		$\cM^{\rm good}\eqdef\cM \setminus \cM^{\rm bad}$.
		Then,
		\begin{align}\label{eq:second-moment-an-sum-good-bad}
			\E[Y^2]&=\sum_{x\in V}\P^{\rm an}\ttonde{\cN_x \cap \cM}=\sum_{x\in V}\P^{\rm an}\ttonde{\cN_x \cap \cM^{\rm good}}+\sum_{x\in V}\P^{\rm an}\ttonde{\cN_x \cap \cM^{\rm bad} } \fstop	\end{align}
		As for the second sum above, we have
		\begin{align}\label{eq:N-M-bad}\sum_{x\in V}\P^{\rm an}\ttonde{\cN_x \cap \cM^{\rm bad} }\le n\times \frac{3}{n^2}\times\( \frac{2t^2}{n}\times \frac{3t^2}{n} + \frac{2t^2}{n}\times \frac{t^2}{n} \)=o\tonde{\frac{1}{n^{2.5}}}\fstop
		\end{align}
		\textcolor{black}{where the factor $n$ comes from the sum and the symmetry of the model, the term $3/n^2$ follows from \cref{eq:nx}, and the term within brackets is an estimate of $\P^{\rm an}\ttonde{ \cM^{\rm bad} \mid \cN_x }$. For the latter we argue as follows: either the walk $Z^{(3)}$ hits one of the trajectories of $Z^{(1)}$ or $Z^{(2)}$ and, subsequently the walk $Z^{(4)}$ ends at the same point as $Z^{(3)}$; or   the walk $Z^{(3)}$ does \emph{not} hit  $Z^{(1)}\cup Z^{(2)}$ and, subsequently, the walk $Z^{(4)}$ hits \emph{both} $Z^{(1)}\cup Z^{(2)}$ \emph{and} $Z^{(3)}$ (which at this point will be disjoint from $Z^{(1)}\cup Z^{(2)}$).}
		For what concerns the first sum, we argue as follows: call $\sigma\in\cN_x$ a realization of the paths of  $Z^{(1)}$ and $Z^{(2)}$ realizing $\cN_x$. For such a $\sigma$, call $\cM^{\rm good}(\sigma)$ the set of paths of $Z^{(3)}$ and $Z^{(4)}$  realizing $\cM$ \emph{and}  not intersecting  $\sigma$. Then, 
		\begin{align*}
			\P^{\rm an}\ttonde{\cN_x \cap \cM^{\rm good}}&=\sum_{\sigma\in\cN_x}\sum_{\eta\in \cM^{\rm good}(\sigma)}\P^{\rm an}(\sigma)\P^{\rm an}(\eta\mid \sigma)\\
			&=\sum_{\sigma\in\cN_x}\sum_{\eta\in \cM^{\rm good}(\sigma)}\P^{\rm an}(\sigma)\P^{\rm an}(\eta)\\
			&\le \P^{\rm an}(\cN_x)\P^{\rm an}(\cM)\fstop
		\end{align*}
		By combining this with  \cref{eq:second-moment-an-sum-good-bad,eq:N-M-bad,eq:N-M}, we get  	 
		\begin{equation}
			\E[Y^2]\le  \E[Y]^2 + o(n^{-2.5})\comma
		\end{equation}
		and, thus, \cref{eq:second-moment}. This concludes the proof of the lemma.
	\end{proof}
	\subsection{Mixing of auxiliary chain}\label{suse:mixing}
	\begin{color}{black}
		This section is devoted to the proof of the fact that  the auxiliary chain mixes, w.h.p., within time $S=\log^3(n)$. 
		More precisely, recalling the event $\cA_4$ in \cref{hp:mixing}, we show:
		\begin{proposition}\label{prop:mix-diag}
			$\lim_{n\to \infty}\P(\cA_4)=1$.
		\end{proposition}
	\end{color}
	Recall that \textcolor{black}{the auxiliary} and the product chain can be perfectly coupled as long as the two walks do not sit on the same vertex. Nevertheless, \textcolor{black}{despite the fact that} the analogue of \cref{prop:mix-diag} for the product chain is an immediate corollary of \cref{eq:cutoff}, establishing this for the auxiliary chain requires a finer analysis on the visits to the diagonal.
	
	We divide the proof of \cref{prop:mix-diag} into several intermediate steps (\cref{lemma:mu-xy-new,lemma:mu+,prop:next-meeting}), and present the concluding arguments at the end of this section.
	
	\begin{color}{black}
		As a first step we show that, conditionally on having a DFA in which $x$ and $x'$ have a common in-neighbor, the probability that  the random DFA has the property that  two independent walks starting at $(x,x')$ meet in a short time is small.
	\end{color}
	
	\begin{lemma}\label{lemma:mu-xy-new}
		For every sequence $(x,x')=(x_n,x_n')\in V^2_{\not=}$,  let
		\begin{equation}\label{eq:Exx}
			\cE_{x,x'}\eqdef\{x \text{ and } x' \text{ have (at least) a common in-neighbor in $G$} \}\fstop
		\end{equation}
		Then, for every $t=t_n\ge 1$ and $\epsilon=\epsilon_{n}>0$,
		\begin{equation}\label{eq:lemma-mu-xy-new}
			\P\ttonde{\Prob_{(x,x')}\ttonde{\tau_{\rm meet}<t}> \epsilon\mid \cE_{x,x'}}\le \frac8\epsilon\:\frac{\log^2(n)\:t^2}{n}\fstop
		\end{equation}
	\end{lemma}
	\begin{proof}
		Note that $\P\ttonde{\emparg\mid \cE_{x,x'}}$ can be sampled as follows:
		\begin{enumerate}
			\item \label{it:an1} To each vertex $y\in V$ attach  two Bernoulli random variables, $W_{x}^y$ and $W_{x'}^y$, having the following joint law:
			\begin{align}\nonumber
				\Pr\(W_{x}^y=0,W_{x'}^y=0 \)&=
				\frac{\binom{n-2}{r}}{\binom{n}{r}} \\
				\label{eq:in-neigh_construction}
				\Pr\(W_{x}^y=1,W_{x'}^y=1 \)&=\frac{\binom{n-2}{r-2}}{\binom{n}{r}}\\
				\nonumber
				\Pr\(W_{x}^y=1,W_{x'}^y=0 \)&=\Pr\(W_{x}^y=0,W_{x'}^y=1 \)=\frac{\binom{n-2}{r-1}}{\binom{n}{r}}\fstop
			\end{align} 
			(Here, \textquotedblleft$W^y_x=1$\textquotedblright\ corresponds to constructing the directed edge $y\to x$ endowed with a random color.)
			\item \label{it:an2} If $W_{x}^{y}+W_{x'}^y\neq 2$ for all $y\in V$, then resample all variables $W$'s, restarting from \cref{it:an1}.
			\item \label{it:an3} For  $y\in V$, if $W_x^y=1$, then connect $y\to x$ and assign this edge a random color, and similarly for $W_{x'}^y$; if $W^y_{x}+W^{y}_{x'}=2$, color the corresponding two edges with two distinct random colors. 
			\item \label{it:an5} Call $\sigma$ the partial environment generated so far (\cref{it:an1,it:an2,it:an3}).
			\item Complete the rest of the random DFA: construct a colored digraph $G'$ with the $n-2$ vertices in $V\setminus\{x,x' \}$, and out-degrees $d_y^+=r-W_{x}^y-W_{x'}^y$ for all $y\in V\setminus\{x,x'\}$.
			\item Call $G=G'\cup \sigma$ the resulting DFA.
		\end{enumerate}
		
		Let $\cF_{x,x'}$ be the event that $\sigma_0=\sigma$ constructed in \cref{it:an1,it:an2,it:an3,it:an5} has no arrows outgoing $x$ nor $x'$. We now show that there exists $C'=C'(r)>0$ such that 
		\begin{align}\label{eq:toprove}
			\P\tonde{\cF_{x,x'}^\complement\mid \cE_{x,x'}}\le\frac{C'}{n}\fstop
		\end{align}
		The proof of \cref{eq:toprove} goes as follows. Let $\{y\not\to z\}$ denote the event that no arrow from $y$ points to $z\in V$; then,
		\begin{align*}\P\tonde{\{x\not\to x\} \cap \{x\not\to x'\} }=\(1-\frac{2}{n} \)\(1-\frac{2}{n-1} \)\cdots\(1-\frac{2}{n-r+1} \)\ge1-\frac{3r}{n}\comma
		\end{align*}
		where the last estimate holds for all $n$ sufficiently large. Therefore, by independence,
		\begin{align}\label{eq:step1}
			\P\tonde{\cF_{x,x'}}\ge \(1-\frac{3r}{n}\)^2\ge 1-\frac{6r}{n}\fstop
		\end{align}
		Recall further that
		\begin{align}\label{eq:step2}
			\P\(\cE_{x,x'}^c \)=\P\tonde{\bigcap_{y\in[n]} \left\{\{y\not\to x\}\cup \{y\not\to x' \}\right\} }=\(1-\frac{\binom{n-2}{r-2}}{\binom{n}{r}} \)^n=1-\Theta(n^{-1})\fstop
		\end{align}
		By  the bound in \cref{eq:step2}, we estimate the right-hand side in \cref{eq:toprove} as follows:	
		\begin{equation}\label{eq:step3}
			\P\(\cF_{x,x'}^c\mid \cE_{x,x'} \)=\frac{\P\( \cF_{x,x'}^c\cap \cE_{x,x'}\)}{\P(\cE_{x,x'})}\le O\(n\)\times \P\( \cF_{x,x'}^c\cap \cE_{x,x'}\)\fstop
		\end{equation}
		As a consequence,  \cref{eq:toprove} holds if we show
		\begin{equation}\label{eq:toprove2}
			\P\( \cF_{x,x'}^c\cap \cE_{x,x'}\)=O(n^{-2})\fstop
		\end{equation}
		To the purpose of proving \cref{eq:toprove2}, introduce  the event
		\begin{equation}
			\cG_{x,x'}=\(\{x\to x \}\cap\{x\to x' \}\) \cup  \(\{x'\to x \}\cap\{x'\to x' \}\)\comma
		\end{equation}
		and write
		\begin{equation}\label{eq:step3bis}
			\P\( \cF_{x,x'}^c\cap \cE_{x,x'}\)= \P\( \cF_{x,x'}^c\cap \cE_{x,x'}\cap \cG_{x,x'}\)+\P\( \cF_{x,x'}^c\cap \cE_{x,x'}\cap \cG_{x,x'}^c\)\fstop
		\end{equation}
		We now bound the two probabilities on the right-hand side above. On the one hand,
		\begin{equation}\label{eq:step4}
			\P\( \cF_{x,x'}^c\cap \cE_{x,x'}\cap \cG_{x,x'}\)= 	\P\( \cE_{x,x'}\cap \cG_{x,x'}\)\le  \P\(  \cG_{x,x'}\) \le 2\frac{\binom{n-2}{r-2}}{\binom{n}{r}}=O(n^{-2})\comma
		\end{equation}
		while, on the other hand,
		\begin{equation}\label{eq:step5}
			\P\( \cF_{x,x'}^c\cap \cE_{x,x'}\cap \cG_{x,x'}^c\)\le \P\( \cF_{x,x'}^c\)\P\( \cE_{x,x'}\mid \cG_{x,x'}^c\)\le \frac{6r}{n} (n-2)\frac{\binom{n-2}{r-2}}{\binom{n}{r}}=O(n^{-2})\fstop
		\end{equation}
		(In the last inequality we used \cref{eq:step1} and a union bound.) By plugging \cref{eq:step4,eq:step5} into \cref{eq:step3bis}, we deduce \cref{eq:toprove2}; by combining this and \cref{eq:step3}, we conclude the proof of  \cref{eq:toprove}.

		We now estimate the right-hand side of \cref{eq:lemma-mu-xy-new}. By \cref{eq:toprove}, we get 
		\begin{align}\label{eq:0}	\begin{aligned}	\P\tonde{\Prob_{(x,x')}\tonde{\tau_{\rm meet}<t}> \epsilon\mid \cE_{x,x'}}&\le \P\tonde{\Prob_{(x,x')}\tonde{\tau_{\rm meet}<t}> \epsilon\mid \cE_{x,x'},\cF_{x,x'}} + \frac{C'}{n}\\
				&\le  \frac{1}{\epsilon}\,	\E\quadre{\Prob_{(x,x')}(\tau_{\rm meet}<t)\mid \cE_{x,x'}, \cF_{x,x'}}+ \frac{C'}{n}\comma
			\end{aligned}
		\end{align}
		where the last step is a consequence of Markov inequality.
		In estimating the expectation on the right-hand side above, we rewrite it as
		\begin{equation}\label{eq:annealing}
			\E\quadre{\Prob_{(x,x')}(\tau_{\rm meet}<t)\mid \cE_{x,x'},\cF_{x,x'}}=\hat\P^{\rm an}(\tau_{\rm meet}^{x,x'}<t\mid \cE_{x,x'},\cF_{x,x'})\comma
		\end{equation}
		where $\hat\P^{\rm an}(\emparg \mid \cE_{x,x'},\cF_{x,x'})$ is the law of a non-Markovian process and $\tau^{x,x'}_{\rm meet}$ random variables constructed as follows:
		\begin{enumerate}[(i)]	
			\item\label{it:2nd-an1} Construct the partial environment $\sigma$ incoming $x$ and $x'$ as described in \cref{it:an1,it:an2,it:an3,it:an5}, and resample it if $\sigma\notin \cF_{x,x'}$.
			\item Start two  walks in $x$ and $x'$ and set $\sigma_0\coloneqq \sigma$.
			\item At each  time-step $s \ge 1$, given the environment $\sigma_{s-1}$, let the first walk choose independently and uniformly at random one of the $r$ colors; if the selected color has already been assigned a target state, then let the particle move to that state; if not, select a target independently and uniformly at random among those states that are not  already targeted by the state the walk sits at. Add that directed edge to the environment $\sigma_{s-1}$, calling this new environment $\sigma_{s-\frac{1}{2}}$. Given the environment $\sigma_{s-\frac{1}{2}}$, perform this same procedure for the second walk  and call $\sigma_s$ the environment finally generated from $\sigma_{s-\frac{1}{2}}$ and this procedure.
			\item Stop the process as soon as the two walks visit the same state at the same \emph{integer} time; call then $\tau^{x,x'}_{\rm meet} \in \N$ this time.
		\end{enumerate}
		
		We now provide an upper bound for the right-hand side of \cref{eq:annealing}. To this purpose, let $Y\coloneqq\{y_{1/2},y_1,\ldots,\ldots, y_{t-1/2},y_t\}$ denote the set of  states visited by  the two walks. Then, in order for the event $\{\tau^{x,x'}_{\rm meet}< t\}$ to occur,  $\abs{Y}<2t$ must hold. In order to estimate the latter event, fix some  $m>0$ and call, for all $j\le 2t$,
		\begin{align}\label{eq:Aj}
			\cJ_j\coloneqq \left\{\abs{\{y_{1/2},y_1,\ldots, y_{j/2}\}} = j  \right\} \bigcap \left\{\{y_{1/2},y_1,\ldots, y_{j/2}\}\cap \sigma = \emp\right\}\bigcap \left\{\abs{\sigma}\le m\right\}\fstop
		\end{align}
		(Here, with a slight abuse of notation, $\sigma$ indicates the vertices with at least one outgoing edge being revealed in \cref{it:2nd-an1}.)
		Since $\cJ_{2t} \subset \cJ_{2t-1} \subset \cdots \subset \cJ_1$, 
		\begin{equation}\label{eq:1ante}
			\begin{split}
				\hat\P^{\rm an}\tonde{\cJ_{2t}\mid \cE_{x,x'},\cF_{x,x'}}&=\hat \P^{\rm an}\tonde{\cJ_{2t}\mid \cJ_{2t-1}, \cE_{x,x'},\cF_{x,x'}}\hat\P^{\rm an}\tonde{\cJ_{2t-1}\mid \cE_{x,x'},\cF_{x,x'}}\\
				&\ge  \tonde{1-\frac{m+2t}{n}} \hat\P^{\rm an}\tonde{\cJ_{2t-1}\mid \cE_{x,x'},\cF_{x,x'}}
				\ge \cdots\\
				&\ge \tonde{1-4 \frac{mt^2}{n}}\hat\P^{\rm an}\tonde{\abs{\sigma} \le m\mid \cE_{x,x'},\cF_{x,x'}}\fstop
			\end{split}
		\end{equation}
		
		We are left to control $|\sigma|$, namely the sum of in-going connections of $x$ and $x'$, conditionally on $\cE_{x,x'}\cap \cF_{x,x'}$. Start by rewriting
		\begin{align}\label{eq:1}
			\hat\P^{\rm an}\tonde{\abs{\sigma} > m\mid \cE_{x,x'},\cF_{x,x'}}= \frac{\hat\P^{\rm an}\tonde{\abs{\sigma} > m,\: \cE_{x,x'},\:\cF_{x,x'}}}{ \hat\P^{\rm an}( \cE_{x,x'},\:\cF_{x,x'} )}\le \frac{\hat\P^{\rm an}\tonde{\abs{\sigma} > m}}{ \hat\P^{\rm an}( \cE_{x,x'},\:\cF_{x,x'} )} \fstop
		\end{align}
		By \cref{eq:toprove,eq:step2}, we have
		\begin{equation}
			\hat\P^{\rm an}( \cE_{x,x'},\:\cF_{x,x'} )=\Omega(n^{-1})\fstop
		\end{equation}
		We are left to estimate  the numerator on the right-hand-side of \cref{eq:1}.
		\textcolor{black}{Note that, without any conditioning, the sum of the in-going connections of $x$ and $x'$  satisfies,
			for all $n$ sufficiently large,}
		\textcolor{black}{
			\begin{align}\label{eq:1bis}
				\hat\P^{\rm an}\tonde{\abs{\sigma} > m}\le \Pr\tonde{{\rm Bin}\(2n,\frac{2 r}{n}\)> m}\,\qquad m\ge 0\fstop
		\end{align}}
		\textcolor{black}{Indeed, the number of in-going connections to any vertex is distributed as ${\rm Bin}(n,r/n)$; moreover, conditionally on the realization of the in-going connections of $x$, the number of in-going connections of $x'$ is dominated by ${\rm Bin}(n,2r/n)$.} 
		\textcolor{black}{
			Taking $m=4 r a$ for some $a=a_n>1$ to be fixed later, and using Chernoff bound, we obtain
			\begin{align}\label{eq:1tris}
				\hat\P^{\rm an}\tonde{\abs{\sigma} > m}\le \exp\(-\frac{4r(a-1)^2}{a+1} \)\fstop
			\end{align}
			Hence, by choosing, e.g., $a=\frac{\log^2(n)}{4r}$ (hence, $m=\log^2(n)$), we finally get, for every $c>0$, 
			\begin{equation}\label{eq:1quater}
				\hat\P^{\rm an}\tonde{\abs{\sigma} > \log^2(n)\mid \cE_{x,x'},\cF_{x,x'}}\le n^{-c}\fstop
		\end{equation}}
		In conclusion, by plugging \cref{eq:1quater} into \cref{eq:1ante}, we deduce
		\begin{align}\label{eq:2}
			\hat\P^{\rm an}\tonde{\cJ_{2t}^\complement\mid \cE_{x,x'},\cF_{x,x'}}\le 5\:\frac{\log^2(n)\:t^2}{n}\fstop
		\end{align}
		By combining $\{\tau_{\rm meet}^{x,x'}<t\}\subset \{|Y|<2t\}$, the definition of $\cJ_{2t}$ in \cref{eq:Aj}, we get
		\begin{align*}
			\hat\P^{\rm an}(\tau_{\rm meet}^{x,x'}<t\mid \cE_{x,x'},\cF_{x,x'})&\le \hat\P^{\rm an}(|Y|<2t\mid \cE_{x,x'},\cF_{x,x'}) \\
			&\le \hat\P^{\rm an}(\cJ_{2t}^\complement\mid \cE_{x,x'},\cF_{x,x'})	\\
			&\quad+\hat\P^{\rm an}(Y\cap \sigma \neq \emp, \abs{\sigma}\le \log^2(n)\mid \cE_{x,x'},\cF_{x,x'})\\
			&\quad+\hat\P^{\rm an}(\abs{\sigma}> \log^2(n)\mid \cE_{x,x'},\cF_{x,x'})  \\
			&\le 5\:\frac{\log^2(n)\:t^2}{n}  + \frac{ \log^2(n) 4t}{n}+n^{-2}\le 7\:\frac{\log^2(n)\:t^2}{n} \comma
		\end{align*}
		\begin{color}{black}
			where in the third line we used \cref{eq:1quater}, \cref{eq:2},  and the following observation: $|Y|\le2t$ and at each step of the construction the probability of creating a connection to $\sigma$ is boudned by $2|\sigma|/n$, for all $n$ large enough.
		\end{color}
		Combining this with \cref{eq:0,eq:annealing} yields the desired result.
	\end{proof}

	In what follows, we will need the following definitions related to the auxiliary chain $\Xi_t$:
	\begin{itemize}
		\item $\tau_\Delta$ ($\in \N$) denotes the first hitting time of the state $\Delta$;
		\item  $\tau_{\Delta,+}$ ($>\tau_\Delta$) denotes the first exit time  from the state $\Delta$ \emph{after} the first visit to $\Delta$;
		\item $\mu_+$ is the distribution on $\tilde V$ of $\Xi_{\tau_{\Delta,+}}$ under $\tilde \Prob_\Delta$. 
	\end{itemize} 
	By definition, $\mu_+(\Delta)=0$. Hence, $\mu_+$ is fully supported on $V^2_{\not=}$, thus, uniquely extends to a probability measure on $V^2$; moreover,
	\begin{equation}\label{eq:mu+}
		\mu_+((x,x'))=\frac{r}{r-1}\frac{\sum_{z\in V} \pi(z)^2\(P(z,x)P(z,x')\)}{\sum_{w\in V}\pi(w)^2}\comma\qquad (x,x')\in V^2_{\not=}\fstop
	\end{equation}
	Further, recalling \cref{eq:Exx}, the support of $\mu_+$ consists of the states $(x,x')\in V^2_{\not=}\subset \tilde V$ for which $\cE_{x,x'}$ holds. Finally, 
	\begin{itemize}
		\item $\varphi:V^2\to \tilde V$ is given by 
		\begin{equation}\label{eq:varphi}
			\varphi((x,x'))\eqdef \begin{cases}
				(x,x') &\text{if}\ (x,x')\in V^2_{\not=}\\
				\Delta &\text{else}\fstop
			\end{cases}
		\end{equation}
	\end{itemize}
	
		\begin{color}{black}
In words, the measure $\mu^+$ represent the exit distribution from the diagonal.
	\end{color}
In \cref{lemma:mu+,prop:next-meeting}, we prove some properties concerning the measure $\mu_+$ and the meeting time of two independent walks when initialized according to $\mu_+$. We start by providing an upper bound for the  maximum of $\mu_+$ which holds w.h.p.. 
	\begin{lemma}[Maximum of $\mu_+$]\label{lemma:mu+}
		W.h.p., 
		\begin{equation}\label{eq:max-mu+}
			\max_{\xi\in \tilde V}\mu_+(\xi)\le \frac{\log^{17}(n)}{n}\fstop
		\end{equation}
	\end{lemma}
	\begin{proof}
		Recall that $\mu_+(\Delta)=0$, hence we estimate $\mu_+$ on $V^2_{\not=}$ only.
		We start by showing that,  w.h.p., all distinct vertices in the original graph have at most two common in-neighbors. Indeed, calling $\cW_{x,x'}$ the event that $x$ and $x'\in V$ have at least three common in-neighbors, by the union bound and the representation employed in \cref{eq:in-neigh_construction}, there exist $c_1,c_2>0$ such that 
		\begin{align}\label{eq:3-in}\begin{aligned}
				\P\tonde{\bigcup_{(x,x')\in V^2_{\not=}} \cW_{x,x'}} &\le n^2\, \P\tonde{\cW_{x,x'	}}\\
				&= n^2 \Pr\left(\Bin\(n,\frac{\binom{n-2}{r-2}}{\binom{n}{r}}\)
				\ge 3 \right) 	\le n^2 \(n\cdot \frac{c_1}{n^2}\)^3\le \frac{c_2}{n}\fstop
			\end{aligned}
		\end{align}	
		Recall \cref{eq:mu+}. 	Then, 		by  \cref{eq:3-in}, $r\ge 2$, and Cauchy-Schwarz inequality $\sum_{w\in V} \pi(w)^2 \ge \frac{1}{n}$,  	
		\begin{align*}
			\P\(\max_{(x,x') \in V^2_{\not=}}\mu_+((x,x'))> \frac{\log^{17}(n)}{n}\)&\le \P\(\max_{z\in V} \pi(z)^2 > \frac{r(r-1)}{3} \frac{\log^{17}(n)}{n} \sum_{w \in V}\pi(w)^2 \) + o(1)\\
			&\le \P\( \max_{z\in V} \pi(z) > \frac{\log^{17/2}(n)}{\sqrt 3 n}\) + o(1)\fstop
		\end{align*}
		The  claim in  \cref{eq:max-pi} yields the desired result.
	\end{proof}

	Recall $P^{(2)}\eqdef (P)^{\otimes 2}=P\otimes P$ from \cref{sec:results}.	 The next  lemma establishes that two independent walks initialized according to $\mu_+$ are, w.h.p., unlikely to meet within a logarithmic time; this carries some implications on the mixing of the auxiliary chain when starting from $\mu_+$. 
	\begin{lemma}\label{prop:next-meeting}
		Let $\beta >0$ and $t\eqdef \log^\beta(n)$.	 Then,  w.h.p., 
		\begin{equation}\label{eq:meet-mu+}
			\sum_{(x,x')\in V^2_{\not=}}\mu_+((x,x'))\Prob_{(x,x')}(\tau_{\rm meet}<t)\le n^{-1/4}\fstop
		\end{equation}
	\end{lemma}	
	\begin{proof}
		For notational convenience, set $\gamma\eqdef n^{-1/4}$.
		Call $B$ the random set of states $(x,x')\in V^2_{\not=}$ for which  $\cE_{x,x'}$ in \cref{eq:Exx} holds; further, let $B_+$, resp.\ $B_-$, denote the states $(x,x')\in B$ satisfying $\Prob_{(x,x')}(\tau_{\rm meet}<t)> \gamma/2$, resp. $\le \gamma/2$. 
		We now estimate the size of the random set $B_+$. To this purpose, recall from \cref{eq:step2} that there exists $c_1=c_1(r)>0$ such that
		\begin{equation}\label{eq:Echi}
			\P\tonde{\cE_{x,x'}}\le  \frac{c_1}{n}\comma\qquad  (x,x')\in V^2_{\not=} \fstop
		\end{equation} 
		Then, by Markov's inequality,  for every $k>0$, \cref{eq:lemma-mu-xy-new,eq:Echi} yield
		\begin{align*}
			\P\tonde{|B_+|> k}&\le\frac{1}{k}\sum_{(x,x')\in V^2_{\not=}}\P\( \cE_{x,x'}\cap \left\{\Prob_{(x,x')}(\tau_{\rm meet}<t)> \gamma/2\right\}\)\\
			&= \frac{1}{k}\sum_{(x,x')\in V^2_{\not=}}\P\tonde{\Prob_{(x,x')}\tonde{\tau_{\rm meet}<t}>\gamma/2\mid \cE_{x,x'}}\P\tonde{\cE_{x,x'}}
			\\
			&\le\frac{16}{k\:\gamma}\:\frac{\log^2(n)\:t^2}{n} \sum_{(x,x')\in V^2_{\not=}} \P\tonde{\cE_{x,x'}}\\
			&
			\le\frac{16\:c_1}{k\:\gamma}\:\log^2(n)\:t^2\fstop
		\end{align*}
		Recall that $t=\log^\beta(n)$ and $\gamma= n^{-1/4}$;  hence, setting $k\eqdef n^{3/4}$
		we get	 
		\begin{equation}
			\label{eq:B+}\P\tonde{|B_+|>n^{3/4}}\le n^{-1/3}\fstop
		\end{equation}
		Recall from Lemma \ref{lemma:mu+} that
		\begin{equation}\label{eq:C-comp}
			\P\tonde{\cD^c}=o(1)\comma\qquad
			\cD\eqdef \left\{	
			\max_{\xi\in \tilde V}\mu_+(\xi)\le \frac{\log^{17}(n)}{n}
			\right\}\fstop
		\end{equation}
		Then, \cref{eq:B+,eq:C-comp} yield
		\begin{align*}
			&\P\tonde{\sum_{(x,x')\in V^2_{\not=}}\mu_+((x,x'))\Prob_{(x,x')}(\tau_{\rm meet}<t)>\gamma}\\&\le n^{-1/3}+o(1)+ \P\(  \{\abs{B_+}\le n^{3/4}\}\cap\cD\cap \left\{\sum_{(x,x')\in V^2_{\not=}}\mu_+((x,x'))\Prob_{(x,x')}(\tau_{\rm meet}<t)>\gamma\right\}\)\fstop 
		\end{align*}
		Note that the probability on the right-hand side above equals zero for all $n$ sufficiently large; this follows by splitting the sum over $V^2_{\not=}$ into one sum over  $B_+$ and one over $B_-$, and using the definitions of $B_+$ and $\cD$. This proves \cref{eq:meet-mu+}, thus concluding the proof of the lemma.
	\end{proof}

	We are finally in good shape to conclude the proof of \cref{prop:mix-diag}. Before entering any details, we provide the reader with some general ideas underlying the proof that the auxiliary chain $\tilde P$ is rapidly mixing, uniformly over the initial position. The goal is to couple the chain $\tilde P$ with the product chain $P^{(2)}$ up to the first hitting of the diagonal. If this occurs after the mixing time of $P^{(2)}$, then the natural coupling ensures mixing for $\tilde P$, too. If the hitting of the diagonal occurs before the mixing of the product chain, then it suffices to analyze the mixing of the chain $\tilde P$ when starting from the measure $\mu_+$ in \cref{eq:mu+}. Here, we exploit \cref{prop:next-meeting}, which ensures that the natural coupling between the two chains succeeds over  polylogarithmic times when starting from $\mu_+$, and this is enough to get to the desired result.

	\begin{proof}[{Proof of \cref{prop:mix-diag}}]
		Recall the definitions of $\tau_\Delta$, $\tau_{\Delta,+}$, $\mu_+$ and $\varphi$ given just above \cref{lemma:mu+}, as well as $S=\log^3(n)$. 
		
		We start by proving the following preliminary result:  w.h.p.,
		\begin{align}\label{eq:mixing-mu+}
			\sup_{t< S}\sup_{A\subset \tilde V}\abs{\sum_{(x,x')\in V^2_{\not=}}\mu_+((x,x'))\tonde{ \tilde P^t((x,x'),A)-  (P^{(2)})^t((x,x'),\varphi^{-1}(A))}} \le n^{-1/4}\fstop
		\end{align}
		Since the paths of the product and auxiliary chains can be coupled until the first hitting time of the diagonal,	the left-hand side of \cref{eq:mixing-mu+} equals
		\begin{align*}
			\sup_{t<S}\sup_{A\subset \tilde V}\abs{\begin{array}{l}\sum_{(x,x')\in V^2_{\not=}}\mu_+((x,x'))\tilde \Prob_{\varphi((x,x'))}\tonde{\Xi_t\in A\comma \tau_\Delta< S}\\
					-\sum_{(x,x')\in V^2_{\not=}}\mu_+((x,x'))\Prob_{(x,x')}\tonde{\mathbf{X}^{(2)}_t\in \varphi^{-1}(A)\comma \tau_{\rm meet}< S}\end{array}}\fstop\end{align*}
		Bounding the absolute value above with the maximum between the two sums and setting $A=\tilde V$ there, since
		\begin{equation}\label{eq:meet-hit}
			\tilde \Prob_{\varphi((x,x'))}\tonde{\tau_\Delta=t}=\Prob_{(x,x')}\tonde{\tau_{\rm meet}=t}\comma\qquad (x,x')\in V^2_{\not=}\comma t \in \N\comma
		\end{equation}
		the claim in \cref{eq:meet-mu+} yields  \cref{eq:mixing-mu+}.
		
		We now turn to the proof of $\P(\cA_4)=1+o(1)$.
		Arguing as in the proof of \cref{eq:mixing-mu+},
		\begin{align}\label{eq:eq0}
			\begin{aligned}
				&\max_{\xi\in \tilde V}\norm{\tilde P^S(\xi,\emparg)-\tilde \pi}_{\rm TV}\le \max_{(x,x')\in V^2}\norm{(P^{(2)})^S((x,x'),\emparg)-\pi^{\otimes 2}}_{\rm TV}\\
				+&\max_{(x,x')\in V^2}\sup_{A\subset \tilde V}\abs{\tilde \Prob_{\varphi((x,x'))}\tonde{\Xi_S\in A\comma \tau_\Delta\le S}-\Prob_{(x,x')}\ttonde{\mathbf X^{(2)}_S\in \varphi^{-1}(A)\comma \tau_{\rm meet}\le S}}\fstop
			\end{aligned}
		\end{align}
		Showing that the first term on the right-hand side above vanishes in probability is an immediate consequence of \cref{eq:cutoff} and $S=\omega(\log(n))$; as for the second term, by the strong Markov property and \cref{eq:meet-hit},
		we get, for every fixed $(x,x')\in V^2$ and $A\subset \tilde V$,
		\begin{align*}
			&\cQ_{x,x'}(A)	\eqdef\abs{\tilde \Prob_{\varphi((x,x'))}\tonde{\Xi_S\in A\comma \tau_\Delta\le S}-\Prob_{(x,x')}\ttonde{\mathbf X^{(2)}_S\in \varphi^{-1}(A)\comma \tau_{\rm meet}\le S}}\\
			&\qquad=\abs{\begin{array}{l}\sum_{t=0}^S \Prob_{(x,x')}\ttonde{\tau_{\rm meet}=t} {\tilde \Prob_\Delta\ttonde{\Xi_{S-t}\in A}}\\
					-\sum_{t=0}^S \sum_{y\in V}\Prob_{(x,x')}\ttonde{\tau_{\rm meet}=t\comma \mathbf X^{(2)}_t=(y,y)}\Prob_{(y,y)}\ttonde{\mathbf X^{(2)}_{S-t}\in \varphi^{-1}(A)}
			\end{array}}\\
			&\qquad=\abs{\sum_{t=0}^S\sum_{y\in V} \Prob_{(x,x')}\ttonde{\tau_{\rm meet}=t\comma \mathbf X^{(2)}_t=(y,y)}\tonde{\begin{array}{l}
						\tilde \Prob_\Delta\ttonde{\Xi_{S-t}\in A}\\
						-
						\Prob_{(y,y)}\ttonde{\mathbf X^{(2)}_{S-t}\in \varphi^{-1}(A)}
			\end{array}}}\fstop
		\end{align*}
		\begin{color}{black}
			We now show that
			\begin{align}\label{eq:eq}
				\begin{aligned}
					\max_{(x,x')\in V^2}\sup_{A\subset \tilde V} \cQ_{x,x'}(A)&\le \sup_{t\le S/2}\norm{\tilde P^{S-t}(\Delta,\emparg)-\tilde \pi}_{\rm TV}\\
					&\qquad+ \sup_{t\le S/2}\max_{y\in V} \norm{(P^{(2)})^{S-t}((y,y),\emparg)-\pi^{\otimes 2}}_{\rm TV}\\
					&\qquad+  \max_{(x,x')\in V^2}	\sum_{t=S/2+1}^S  \Prob_{(x,x')}\tonde{\tau_{\rm meet}=t}\fstop
				\end{aligned}
			\end{align}
			Indeed, for the second half of the sum, by moving the absolute value inside the summation and bounding by 1 the difference between round brackets, we obtain
			\begin{align*}
				&\abs{\sum_{t=S/2+1}^S\sum_{y\in V} \Prob_{(x,x')}\ttonde{\tau_{\rm meet}=t\comma \mathbf X^{(2)}_t=(y,y)}\tonde{\begin{array}{l}
							\tilde \Prob_\Delta\ttonde{\Xi_{S-t}\in A}\\
							-
							\Prob_{(y,y)}\ttonde{\mathbf X^{(2)}_{S-t}\in \varphi^{-1}(A)}
				\end{array}}}\\
				&\le\sum_{t=S/2+1}^S  \Prob_{(x,x')}\tonde{\tau_{\rm meet}=t}\comma
			\end{align*}
			which, after taking the supremum over $(x,x')$, corresponds to the last term on the right-hand side of \cref{eq:eq}.
			On the other hand, for the first half of the sum, estimating uniformly in $t\le S/2$ and $y\in V$ the terms inside the round brackets, we get
			\begin{align*}
				&\abs{\sum_{t=0}^{S/2}\sum_{y\in V} \Prob_{(x,x')}\ttonde{\tau_{\rm meet}=t\comma \mathbf X^{(2)}_t=(y,y)}\tonde{\begin{array}{l}
							\tilde \Prob_\Delta\ttonde{\Xi_{S-t}\in A}\\
							-
							\Prob_{(y,y)}\ttonde{\mathbf X^{(2)}_{S-t}\in \varphi^{-1}(A)}
				\end{array}}}\\
				&\le\sup_{t\le \frac{S}2}\sup_{y\in V}\abs{
					\tilde \Prob_\Delta\ttonde{\Xi_{S-t}\in A}
					-
					\Prob_{(y,y)}\ttonde{\mathbf X^{(2)}_{S-t}\in \varphi^{-1}(A)}}\fstop
			\end{align*}
			Finally, adding and subtracting $\tilde{\pi}(A)$ inside the latter absolute value, using the triangle inequality, and taking the supremum over $A\subseteq\tilde{V}$ yields  the first two terms on the right-hand side of \cref{eq:eq}.
		\end{color}
		The second term in \cref{eq:eq} is dealt with as the first one in \cref{eq:eq0}. (There, we employ the fact that $S-t\ge S/2-1=\omega(\log(n))$.) As for the third term in \cref{eq:eq}, 
		\begin{align*}
			\max_{(x,x')\in V^2}	\sum_{t=S/2+1}^S  \Prob_{(x,x')}\tonde{\tau_{\rm meet}=t}&\le S \sup_{t> S/2}\max_{(x,x')\in V^2}  \Prob_{(x,x')}\ttonde{\mathbf X^{(2)}_t\in \Delta}\\
			&\le S \sup_{t> S/2} \max_{(x,x')\in V^2} \norm{(P^{(2)})^{t}((x,x'),\emparg)-\pi^{\otimes 2}}_{\rm TV} + S \tilde \pi(\Delta)\fstop
		\end{align*}
		Since $S= \log^3(n)$, \cref{eq:cutoff} ensures that 
		\begin{align*}
			S \sup_{t> S/2} \max_{(x,x')\in V^2} \norm{(P^{(2)})^{t}((x,x'),\emparg)-\pi^{\otimes 2}}_{\rm TV}\overset{\P}\to 0\comma
		\end{align*} while 	$S \tilde \pi(\Delta)\overset{\P}\to 0$ by \cref{lemma:A2} (cf.\ \cref{hp:pidelta}). 
		
		We are now left with showing that the first term on the right-hand side of \cref{eq:eq} vanishes in probability. Recalling the definition of $\tau_{\Delta,+}$, note that, 	for any given DFA, under $\tilde \Prob_\Delta$ {the stopping time	 $\tau_{\Delta,+}$} is  distributed as a {geometric distribution of success probability $p=1-\frac{1}{r}$}:
		\begin{align}\label{eq:geometric}
			\tau_{\Delta,+}\sim {\rm Geom}\tonde{1-\frac{1}{r}}\comma \qquad \text{under}\ \tilde \Prob_\Delta\fstop\end{align} Indeed, when attempting to jump, the process associated to $\tilde P$ stays on $\Delta$ if the second coordinate chooses the same arrow that the first one chose, and this occurs with probability $1/r$,  independently at each step. (Recall that, for any given $x, y\in V$, multiple directed edges $x\to y$ are not allowed, and this fact holds regardless of connectedness properties of the graph.)
		Hence, setting $\hslash\eqdef \log \log (n)$, the strong Markov property and the triangle inequality yield
		\begin{align}\label{eq:final}
			\begin{aligned}
				&\sup_{t\le S/2}\norm{\tilde P^{S-t}(\Delta,\emparg)-\tilde \pi}_{\rm TV}
				\\
				&\le \tilde \Prob_\Delta\ttonde{\tau_{\Delta,+}\ge \hslash} + \sup_{S/2-\hslash\le t\le S }\norm{\mu_+\tilde P^{t-\hslash}-\tilde \pi}_{\rm TV}\\
				&\le \tilde \Prob_\Delta\ttonde{\tau_{\Delta,+}\ge \hslash} + \sup_{S/2-\hslash\le t\le S} \max_{(x,x')\in V^2}\norm{(P^{(2)})^{t-\hslash}((x,x'),\emparg)-\pi^{\otimes 2}}_{\rm TV}\\
				& + \sup_{S/2-\hslash\le t\le S }  \sup_{A\subset \tilde V}\abs{\sum_{(x,x')\in V^2_{\not=}}\mu_+((x,x'))\tonde{ \tilde P^{t-\hslash}((x,x'),A)-  (P^{(2)})^{t-\hslash}((x,x'),\varphi^{-1}(A))}}\fstop
			\end{aligned}
		\end{align}
		The first term on the right-hand side of \cref{eq:final} vanishes $\P$-a.s.\ since $\hslash$ is diverging and $\tau_{\Delta,+}$ is geometric with constant parameter; the second  and third terms vanish in probability by applying, respectively, \cref{eq:cutoff} with  $S/2-\hslash=\omega(\log(n))$, and \cref{eq:mixing-mu+} with $S=\log^3(n)$. This concludes the proof of the proposition.
	\end{proof}

	\subsection{Number of returns}\label{suse:returns}
	In this section, we provide a first order estimate for the expected number of returns to the diagonal \textcolor{black}{within a time $T=\log^5(n)$}. To this purpose, recall the definition of $\cA_5=\cA_5(\eps)$ in \cref{hp:R}, and define
	\begin{align}
		\tilde R(\Delta)\eqdef \sum_{t=0}^T \tilde P^t(\Delta,\Delta)\fstop
	\end{align}
	\begin{proposition}\label{prop:R} $\lim_{n\to \infty}\P\tonde{\cA_5}=1$, for every $\eps>0$.
	\end{proposition}
	\begin{proof}
		Recall that, for any given DFA $G$ and under $\tilde \Prob_\Delta$, $\tau_{\Delta,+}$ is geometric with parameter $1-\frac{1}{r}$ (cf.\ \cref{eq:geometric}).
		Therefore, estimating from below $\tilde P^t(\Delta,\Delta)$ with  
		\begin{align*}\tilde\Prob_\Delta\(\Xi_s = \Delta\ \text{for all}\ s \in \{0,\ldots, t\} \)=\tonde{\frac{1}{r}}^t\comma
		\end{align*} we get, since $T$ diverges as $n\to \infty$, 	
		\begin{equation}\label{eq:lbR}
			\tilde R(\Delta)\ge \sum_{t=0}^T\(\frac{1}{r}\)^t= \frac{r}{r-1}+o(1)\comma \qquad \P\text{-a.s.}\fstop
		\end{equation}

		On the other hand, for any given  $G$, we have
		\begin{align*}
			\tilde R(\Delta)&\le \sum_{t=0}^T \tilde\Prob_\Delta\(\Xi_s=\Delta\ \text{for all}\ s\in \{0,\ldots, t\}\)\\
			&+  \sum_{t=0}^T \tilde\Prob_\Delta\(\exists s, \tilde s \in \{1,\ldots, t\}, s< \tilde s:   \Xi_{s} \neq \Delta, \Xi_{\tilde s}=\Delta\) \\
			&\le \frac{r}{r-1} + T \sum_{\xi\in V^2_{\not=}} \mu_+(\xi)\tilde\Prob_\xi\(\tau_\Delta < T \)\fstop	\end{align*}
		By \cref{eq:meet-hit}, the choice of $T=\log^5(n)$ and  Proposition \ref{prop:next-meeting}, we obtain, for every $\eps > 0$,  
		\begin{equation}
			\P\(\sum_{\xi \in V^2_{\not=}} \mu_+(\xi)\tilde \Prob_\xi\(\tau_\Delta< T\)>\frac{\eps}{T} \) \underset{n\to \infty}\longrightarrow 0\comma 
		\end{equation}
		and, thus,
		\begin{equation}\label{eq:ubR}
			\P\(\tilde R(\Delta) > \frac{r}{r-1}+\eps\)\underset{n\to \infty}\longrightarrow 0\fstop
		\end{equation}
		Combining  \cref{eq:lbR} and \cref{eq:ubR} yields the desired claim.
	\end{proof}

	\section{Proofs of  main results}\label{sec:theorems_proofs}
	This section contains the proofs of \cref{th:1,th:2}.
	\subsection{Proof of \cref{th:1}}
	As a consequence of \cref{pr:fvtltilde}, for every $\eps>0$, w.h.p.,   
	\begin{align*}
		\sup_{t\ge 0} \frac{\Prob_{\pi\otimes \pi}\tonde{\tau_{\rm meet}>t}}{\tonde{1-\Lambda}^t}= \sup_{t\ge 0} \frac{\max_{x,y\in V}\Prob_{(x,y)}\tonde{\tau_{\rm meet}>t}}{\tonde{1-\Lambda}^t}\frac{\Prob_{\pi\otimes \pi}\tonde{\tau_{\rm meet}>t}}{\max_{x,y\in V}\Prob_{(x,y)}\tonde{\tau_{\rm meet}>t}}<1+\eps\fstop
	\end{align*}
	Hence, it suffices to show that, for every $\eps>0$, w.h.p., 
	\begin{align}\label{eq:claim-thm1.1}
		\sup_{t\ge 0}	\frac{\max_{x,y\in V}\Prob_{(x,y)}\tonde{\tau_{\rm meet}>t}}{\Prob_{\pi\otimes \pi}\tonde{\tau_{\rm meet}>t}}<1+\eps\fstop
	\end{align}
	Let $T\eqdef \log^5(n)$; then, by \cref{lemma:hpfvtl} and \cref{rmk:general}  for every  $\eps>0$, w.h.p., 
	\begin{align}\label{eq:max-MQS21}
		\sup_{t\ge 0}	\frac{\max_{x,y\in V}\Prob_{(x,y)}\tonde{\tau_{\rm meet}>t}}{\Prob_{\pi\otimes \pi}\tonde{\tau_{\rm meet}>t}}<1+\eps\fstop
	\end{align}
	Further, by \cref{pr:fvtltilde}, uniformly over $t\ge 0$, w.h.p., \begin{align}\label{eq:expression1}\Prob_{\pi\otimes\pi}\tonde{\tau_{\rm meet}>t}=(1+o(1)) (1-\Lambda)^{t}\comma
	\end{align}
	yielding \cref{eq:claim-thm1.1}. \qed

	\subsection{Proof of \cref{th:2}}
	In view of \cref{th:1}, it suffices to prove that, for every $(x,y)\in V^2_{\not=}$ and $\eps>0$, w.h.p.,
	\begin{align}\label{eq:inf1}
		\inf_{t\ge 0}\frac{ \Prob_{(x,y)}\tonde{\tau_{\rm meet}>t}}{\tonde{1-\Lambda}^t}>1-\eps\fstop
	\end{align}
	Splitting the  infimum above into two parts and recalling $n\Lambda\overset\P\to 1$ (\cref{th:1}), the claim in \cref{eq:inf1} follows if, for some $s=o(n)$ and  every $\eps>0$, w.h.p.,
	\begin{align}\label{eq:inf2}
		\Prob_{(x,y)}\tonde{\tau_{\rm meet}>s}>1-\eps\comma\qquad
		\inf_{t> s}\frac{\Prob_{(x,y)}\tonde{\tau_{\rm meet}>t}}{\tonde{1-\Lambda}^t}>1-\eps\fstop
	\end{align}
	In what follows, we  prove the two claims in \cref{eq:inf2} with $s=\log^{5}(n)$. (Note that, by \cref{eq:cutoff} from \cref{th:single_rw}, this choice guarantees that
	\begin{align}\label{eq:mix-product}
		\max_{(x,y)\in V^2}\max_{(u,v)\in {\rm supp(\pi^{\otimes 2})}}\abs{\frac{\Prob_{(x,y)}\tonde{\mathbf X^{(2)}_s=(u,v)}}{\pi^{\otimes 2}(u,v)}-1}\le\frac{\eps}{2} 
	\end{align}
	holds w.h.p..)

	As for the first claim in \cref{eq:inf2},  Markov inequality yields
	\begin{align}\label{eq:markov}
		\P\tonde{\Prob_{(x,y)}\tonde{\tau_{\rm meet}\le s}\ge \eps}\le \eps^{-1}\E\quadre{\Prob_{(x,y)}\tonde{\tau_{\rm meet}\le s}}\comma\qquad \eps>0\fstop
	\end{align}
	We now estimate the above expectation by means of an \emph{annealing argument} \textcolor{black}{in the same spirit of that in} the proof of \cref{lemma:A2}: first construct the partial environment generated by the trajectory of length $s$ of the walk starting at $x$; then, conditioning on this path,  construct a path of the same length starting at $y$. Letting $(X_0=x,X_1,\dots, X_s)$ and $(Y_0=y,Y_1,\dots,Y_s)$ denote these two paths, we have
	\begin{equation}\label{eq:thelatter}
		\E\quadre{\Prob_{(x,y)}\tonde{\tau_{\rm meet}\le s}}\le \P^{\rm an}(\{X_0,X_1,\dots, X_s\}\cap\{Y_0,Y_1,\dots,Y_s\} \neq \emp) \le \frac{s^2}{n}\fstop
	\end{equation}
	By plugging \cref{eq:thelatter} into \cref{eq:markov}, the choice  $s=\log^{5}(n)$ ensures the validity of the first claim in \cref{eq:inf2}.
	
	Concerning the second claim in \cref{eq:inf2}, we get, $\P$-a.s.\ and for every $t> s$, 
	\begin{align}\label{eq:delay}
		\Prob_{(x,y)}(\tau_{\rm meet}>t)=\sum_{(u,v)\in V^2_{\not=}}\Prob_{(x,y)}\tonde{\mathbf X^{(2)}_s=(u,v)\comma\tau_{\rm meet}>s} \Prob_{(u,v)}(\tau_{\rm meet}> t-s)\fstop
	\end{align}
	We now claim that there exists   $\nu=\nu_{x,y}^s:V^2\to [0,1]$ such that, for every $\eps>0$, w.h.p.,
	\begin{equation}\label{eq:delay2}
		\Prob_{(x,y)}\tonde{\mathbf X^{(2)}_s=(u,v)\comma\tau_{\rm meet}>s}\ge \(1-\frac\varepsilon2 \)\pi(u)\pi(v) - \nu(u,v) \comma \qquad (u,v)\in V^2_{\not=}\comma
	\end{equation}
	and
	\begin{equation}\label{eq:delay4}
		\sum_{(u,v)\in V^2_{\not=}}\nu(u,v)\le \frac\varepsilon2\fstop
	\end{equation}
	Indeed, letting
	\begin{align*}
		\nu(u,v)\eqdef  \Prob_{(x,y)}\tonde{\mathbf X^{(2)}_s=(u,v)\comma\tau_{\rm meet}\le s}\comma\qquad (u,v)\in V^2_{\not=}\comma
	\end{align*}
	\cref{eq:delay4} follows at once from $\sum_{(u,v)\in V^2}\nu(u,v)= \Prob_{(x,y)}\tonde{\tau_{\rm meet}\le s}$ and the first claim in \cref{eq:inf2} (with $\eps/2$ in place of $\eps$), while
	\cref{eq:mix-product} ensures that, w.h.p.,
	\begin{align*}
		\Prob_{(x,y)}\tonde{\mathbf X^{(2)}_s=(u,v)\comma\tau_{\rm meet}>s}&= \Prob_{(x,y)}\tonde{\mathbf X^{(2)}_s=(u,v)}-\Prob_{(x,y)}\tonde{\mathbf X^{(2)}_s=(u,v)\comma\tau_{\rm meet}\le s}\\
		&\ge \tonde{1-\frac{\eps}{2}}\pi(u)\pi(v) -\nu(u,v)\comma\qquad (u,v)\in V^2_{\not=}\fstop
	\end{align*}
	This proves \cref{eq:delay2}.

	In view of the two assertions in \cref{eq:delay2,eq:delay4}, we are now ready to prove the second claim in \cref{eq:inf2}:
	by plugging \cref{eq:delay2} into \cref{eq:delay} and applying \cref{eq:delay4}, we get, w.h.p.,
	\begin{align}\label{eq:delay3}
		\Prob_{(x,y)}(\tau_{\rm meet}>t)&\ge \sum_{(u,v)\in V^2_{\not=}}\[\(1-\frac\varepsilon2 \)\pi(u)\pi(v)-\nu(u,v) \]\Prob_{(u,v)}(\tau_{\rm meet}> t-s)\\
		\nonumber	&\ge \(1-\frac\varepsilon2 \)\Prob_{\pi\otimes\pi}(\tau_{\rm meet}>t-s) - \frac\varepsilon2 \max_{(u,v)\in V^2}\Prob_{(u,v)}(\tau_{\rm meet}>t-s)\\
		\nonumber	&\ge (1-2\varepsilon)(1-\Lambda)^t,
	\end{align}
	where the last estimate follows by \cref{pr:fvtltilde},  \cref{eq:max-MQS21} and the fact that $s=o(n)$. This proves the second claim in \cref{eq:inf2}, thus, concluding the proof of the theorem. \qed
	
	\appendix
	\begin{color}{black}
	\section{Proof of the FVTL}\label{apx:FVTL}
	This section is devoted to the proof of \cref{fvtl}; hence, the setting and assumptions in \cref{fvtl} are in force all throughout.
	 
Let us briefly recall that $Q=Q_N$ denotes the  transition matrix of a discrete-time irreducible Markov chain --- which we call  $(X_t)_{t\ge0}=(X_t^N)_{t\ge 0}$ ---  on $[N]$ with unique stationary distribution $\mu=\mu_N$, while $\partial\in {\rm supp}(\mu)\subseteq[N]$  represents our target state. Moreover, for every probability distribution $\nu$ on $[N]$, we let $\Q_\nu$ denote the law of chain started at $\nu$, and $\E_{\nu}$ the corresponding expectation; if $\nu=\delta_x$, we simply write $\Q_x$ and $\E_x$. Furthermore, the mixing time $t_{\rm mix}=t_{\rm mix}(Q)$ is defined as in \cref{eq:def-tmix}, and, we observe that  the following estimate for the $L^\infty$-distance-to-equilibrium for the Markov chain $Q$ holds:  for every $N \in \N$ and for every $T$ as in  \cref{eq:hp},
		\begin{align}\label{eq:def-T-linfty}
			\max_{x\in [N]}\left\|Q^T(x,\cdot)-\mu\right\|_{\rm TV}\le	\max_{\substack{x\in [N]\\
					y\in {\rm supp}(\mu)}}\left|\frac{Q^T(x,y)}{\mu(y)}-1 \right|\le \frac{1}{N}\fstop
		\end{align}

We start by recalling a  result by D.\ Aldous \cite{aldous82} (see Eqs.\ (2.1), (2.2) and (2.8), as well as Lemma 2.9 and Remark 2.18), which actually holds for a general Markov chain.
	\begin{proposition}[\cite{aldous82}]\label{lemma-aldous}
There exists a couple $(\mu_\star,\lambda_\star)=(\mu_{\star,N},\lambda_{\star,N})$, where $\lambda_\star\in(0,1)$ and $\mu_{\star}$ is a probability distribution on $[N]\setminus\{\partial\}$, satisfying 
		\begin{equation}
			\lim_{t\to\infty}\Q_{\mu}(X_t=x\mid \tau_{\partial}>t)=\mu_{\star}(y)\comma\qquad y\in[N]\setminus\{\partial\}\comma
		\end{equation}
		and
		\begin{equation}\label{eq:exp}
			\Q_{\mu_\star}(\tau_{\partial}>t )=(1-\lambda_{\star})^t\comma \qquad t \in \N\fstop
		\end{equation}
		Moreover,
		\begin{equation}\label{eq:bound-aldous}
			\left|\frac{\E_{\mu_\star}[\tau_{\partial}]}{\E_{\mu}[\tau_{\partial}]}-1 \right| \le \frac{20}{3}\ \frac{t_{\rm mix} (2+ \log(\E_\mu[\tau_\partial]))}{\E_\mu[\tau_\partial] }\fstop 
		\end{equation}
	\end{proposition}
	We divide the proof of  \cref{fvtl} into three auxiliary lemmas.
	 For the rest of this section, we will assume that \cref{eq:small-mixing} holds true and that the sequence $T=T_N$ satisfies \cref{eq:hp}.
	\begin{lemma}\label{lemma:abdu}
	Recalling that $R=R_{N,T}\eqdef \sum_{t=0}^t Q^t(\partial,\partial)$, we have
		\begin{equation}
			\lim_{N\to \infty}\frac{\E_\mu[\tau_{\partial}]}{R/\mu(\partial)}=1\fstop
		\end{equation}
		In particular, since $R\ge 1$, by \cref{eq:hp}, we have
		\begin{equation}\label{eq:tau-large}
			\liminf_{n\to \infty}\mu(\partial)\ \E_\mu[\tau_{\partial}]\ge 1\fstop
		\end{equation}
	\end{lemma}
	\begin{proof}
		By \cite[Lemma 2.1]{aldous-fill-2014}, we have
		$$\E_\mu[\tau_\partial]=\frac{Z(\partial,\partial)}{\mu(\partial)}\comma$$
		where $Z$ is the so called \emph{fundamental matrix} defined as
		\begin{equation}
			Z(x,y):=\sum_{t=0}^{\infty}\tonde{Q^t(x,y)-\mu(y)}\comma\qquad x,y\in[N]\fstop
		\end{equation}
		Observe that
		\begin{equation}
			Z(\partial,\partial)=\(\sum_{t=0}^{T}Q^t(\partial,\partial)\)-\tonde{T+1}\mu(\partial)+\sum_{t>T}\tonde{Q^t(\partial,\partial)-\mu(\partial)}=R\left(1+o(1)\right)\comma
		\end{equation}
		where in the last equality we used \cref{eq:hp}, $R\ge 1$, \cref{eq:def-T-linfty}, and the submultiplicativity of the $L^{\infty}$-distance.
	\end{proof}
\begin{lemma}\label{coro-aldous}
The quantity  $\lambda_\star\in (0,1)$ in \cref{lemma-aldous} satisfies
	\begin{equation}
	(1-	\lambda_{\star})^T=1+o(1	)\fstop
	\end{equation}
\end{lemma}
\begin{proof}
	By \cref{eq:exp}, we have
	\begin{equation}
		\E_{\mu_{\star}}[\tau_{\partial}]=\frac{1}{\lambda_\star}\fstop
	\end{equation}
Then, we obtain $\lambda_\star T=o(1)$ by  \cref{eq:tau-large,eq:bound-aldous,eq:hp}, from which the result follows.
\end{proof}
	
	\begin{lemma}\label{lemma:eq-is-far}
	
		\begin{equation}\label{eq:eq-is-far}
			\lim_{N\to \infty}\sup_{k\in\N}\frac{\max_{x\in [N]}\Q_x(\tau_{\partial}>kT)}{\Q_\mu(\tau_{\partial}>kT)}=1\fstop
		\end{equation}
	\end{lemma}
	\begin{proof}
		Clearly, it suffices to prove that the limit in \cref{eq:eq-is-far} is $\le 1$, because the other inequality is trivial. Moreover, by a union bound and \cref{eq:hp}.
		\begin{equation}\label{eq:union-b}
			\Q_{\mu}(\tau_{\partial}>T)\ge 1-(T+1)\mu(\partial)=1+o(1)\fstop
		\end{equation}
		Hence, we restrict the attention to $k\ge 2$.  By the strong Markov property, we get, for all $x \in [N]$,
		\begin{equation}\label{eq:cite}
			\begin{split}
			\Q_x(\tau_\partial>kT)&=\sum_{\substack{y \in {\rm supp}(\mu)\\
					y\neq   \partial}}\Q_x(X_T=y,\:\tau_\partial>T)\,\Q_y(\tau_\partial>(k-1)T)\\
				&\qquad\qquad +\sum_{y \not\in {\rm supp}(\mu)}\Q_x(X_T=y)\,\Q_y(\tau_\partial>(k-1)T)\\
			&\le\sum_{y\in {\rm supp}(\mu)}\Q_x(X_T=y)\,\Q_y(\tau_\partial>(k-1)T)\\
			&\qquad\qquad+\sum_{y \not\in {\rm supp}(\mu)}\Q_x(X_T=y)\,\Q_y(\tau_\partial>(k-1)T)\\
			&=\left(1+o(1)\right)\Q_\mu(\tau_\partial>(k-1)T)+\sum_{y \not\in {\rm supp}(\mu)}\Q_x(X_T=y)\,\Q_y(\tau_\partial>(k-1)T)\comma
			\end{split}
		\end{equation}
		where in the third step we used \cref{eq:def-T-linfty}. We can bound the last sum in \cref{eq:cite} by
		\begin{equation}\label{eq:cite2}
			\begin{split}
			\sum_{y \not\in {\rm supp}(\mu)}\Q_x(X_T=y)&\,\Q_y(\tau_\partial>(k-1)T)\\&\le  \max_{x\in[N]}\Q_x(X_T\not\in{\rm supp}(\mu))\ \max_{y\in [N]}\Q_{y}(\tau_\partial>(k-1)T)\\
			&\le \frac1N \max_{y\in [N]}\Q_{y}(\tau_\partial>(k-1)T)\comma
			\end{split}
		\end{equation}
	where in the last step we used again \cref{eq:def-T-linfty}.
	Now call
	\begin{equation}
		f_k\coloneqq \max_{x\in [N]}\Q_x(\tau_{\partial}>kT)\comma \qquad g_k\coloneqq \Q_\mu(\tau_{\partial}>kT)\comma
	\end{equation}
and notice that, plugging \cref{eq:cite2} into \cref{eq:cite} and  taking the maximum over $x\in[N]$, we obtain
\begin{equation}
	f_k\le (1+o(1))g_{k-1}+\frac{1}{N}f_{k-1}\fstop
\end{equation}
It follows by iteration that, for all $N$ sufficiently large,
\begin{equation}\label{eq:iter}
	f_k\le (1+o(1))g_{k-1}+2\sum_{j=1}^{k-1} \frac{1}{N^j}\ g_{k-1-j}\fstop
\end{equation}
		Thanks to \cite[Lemma 3.6]{manzo_quattropani_scoppola_2021}, we also have
		\begin{equation}\label{eq:mqs}
			\lim_{N\to \infty}\sup_{k\in\N}\frac{\Q_\mu(\tau_\partial>(k-1)T)}{\Q_\mu(\tau_\partial>kT)}=\lim_{N\to \infty}\sup_{k\in\N}\frac{g_{k-1}}{g_{k}}=1\comma
		\end{equation}
	hence, for all $N$ sufficiently large,
	\begin{equation}
		g_{j-1}\le 2 g_{j}\quad\Longrightarrow\quad g_{j}\le 2^{k-j-1}g_{k-1}\comma
	\end{equation}
therefore
\begin{equation}
\sum_{j=1}^{k-1} \frac{1}{N^j}\ g_{k-1-j}\le g_{k-1}\sum_{j=1}^{k-1}  \(\frac{2}{N}\)^{j}=o(g_{k-1})\comma
\end{equation}
		from which, together with \cref{eq:iter,eq:mqs}, the desired claim follows.
	\end{proof}
	
	\begin{proof}[Proof of \cref{fvtl}]
		Let $\lambda_\star\in (0,1)$ be as in \cref{lemma-aldous}. By \cref{coro-aldous} and \cref{eq:union-b},
 it is enough to focus on the case $t\ge T$. Moreover, again by \cref{coro-aldous} and by the monotonicity in $t$ of the probabilities under consideration, it suffices to check the validity of \cref{eq:geometric0} for $t=kT$ with $k\ge 2$.
		
		First we prove the lower bound:
		\begin{equation}\label{eq:fvtl-lb}
			\liminf_{N\to\infty} \inf_{k\ge 2}\frac{\Q_{\mu}(\tau_{\partial}>kT)}{(1-\lambda_\star)^{kT}}\ge 1\fstop
		\end{equation}
		Note that, for all $x\in[N]$, we have
		\begin{equation}\label{eq:bound-evo}
			\Q_{\mu_\star}(X_T=x)= (1-\lambda_\star)^T\mu_\star(x)+\lambda_\star\ \sum_{s=1}^T(1-\lambda_\star)^{s-1}\,\Q_{\partial}(X_{T-s}=x)
			\ge (1-\lambda_\star)^T \mu_\star(x)\fstop
		\end{equation}
		Furthermore, by \cref{eq:def-T-linfty,eq:bound-evo,coro-aldous},
		\begin{align*}
			\Q_{\mu}(\tau_\partial>kT)&=\(1+o(1)\) \Q_{\mu_\star Q^T }(\tau_\partial>kT)\\
			&\ge \tonde{1+o(1)} (1-\lambda_\star)^T \Q_{\mu_\star}(\tau_\partial>kT)\\
			&=\tonde{1+o(1)}(1-\lambda_\star)^{(k+1)T}\\
			&=\tonde{1+o(1)} (1-\lambda_\star)^{kT}\comma
		\end{align*}
		and \cref{eq:fvtl-lb} follows. 
		
		We now show the upper bound:
		\begin{equation}\label{eq:fvtl-ub}
			\limsup_{N\to\infty} \sup_{k\ge 2}\frac{\Q_{\mu}(\tau_{\partial}>kT)}{(1-\lambda_\star)^{kT}}\le 1\fstop
		\end{equation}
		For every $x\in[N]$, we have (cf.\ \cref{eq:bound-evo})
		\begin{equation}\label{eq:ub}
			\Q_{\mu_\star}(X_T=x)
			\le (1-\lambda_\star)^T\mu_\star(x)+\lambda_\star\ \E_\partial[L_T(x)]\comma
		\end{equation}
		where $L_T(x)$ denotes the local time spent by the chain in the state $x$ within time $T$, i.e.,
		\begin{equation}
			L_T(x):=\sum_{s=1}^T\ind(X_t=x)\fstop
		\end{equation}
		Clearly, 
		\begin{equation}\label{eq:blue}
			\sum_{x\in[N]}\E_\partial[L_T(x)]= T\fstop
		\end{equation}
		Therefore, 
		\begin{align*}
		&\Q_\mu(\tau_\partial>kT)\\
			\text{\cref{eq:def-T-linfty}}\Longrightarrow		&\qquad=\tonde{1+o(1)} Q_{\mu_\star Q^T}(\tau_\partial>kT)\\
			\text{\cref{eq:ub}}\Longrightarrow	&\qquad\le\tonde{1+o(1)} \sum_{x\in {\rm supp}(\mu)	}(1-\lambda_\star)^T\mu_\star(x)\,\Q_x(\tau_\partial>kT)\\
			&\qquad\qquad+
			\tonde{1+o(1)} \lambda_\star\ \sum_{x\in {\rm supp}(\mu)}\E_\partial[L_T(x)]\,\Q_x(\tau_\partial>kT)\\
			\text{\cref{eq:blue}}\Longrightarrow&\qquad	\le \tonde{1+o(1)} (1-\lambda_\star)^{(k+1)T}+\tonde{1+o(1)} \lambda_{\star}T\ \max_{x\in {\rm supp}(\mu)}\Q_x(\tau_\partial>kT)\\
			\text{ \cref{lemma:eq-is-far,coro-aldous}}\Longrightarrow&\qquad	= (1+o(1))(1-\lambda_\star)^{kT}+o\(\Q_\mu(\tau_\partial>kT)\)\comma
		\end{align*}
		from which \cref{eq:fvtl-ub} follows. This concludes the proof of \cref{fvtl}.
	\end{proof}
\begin{remark}\label{rmk:general}
	\textit{A posteriori}, thanks to \cref{eq:geometric0} in \cref{fvtl}, the claim of \cref{lemma:eq-is-far} generalizes as follows:
	\begin{equation}\label{eq:eq-is-far-2}
		\lim_{N\to \infty}\sup_{t\ge 0}\frac{\max_{x\in [N]}\Q_x(\tau_{\partial}>t)}{\Q_\mu(\tau_{\partial}>t)}=1\fstop
	\end{equation}
\end{remark}
\end{color}

	\subsection*{Acknowledgments}
	The authors wish to thank  Guillem Perarnau for  pointing out the reference \cite{FR2017}, and the anonymous referees for their careful reading of our manuscript.
	During an early stage of this work, M.Q.\ was supported by the European Union's Horizon 2020 research and innovation programme under the Marie Sk\l{}odowska-Curie grant agreement no.\ 945045, and by the NWO Gravitation project NETWORKS under grant no.\ 024.002.003. Moreover, M.Q. thanks the German Research Foundation (project number 444084038, priority program SPP2265) for financial support.
	F.S.\ 	gratefully acknowledges funding by the Lise Meitner fellowship, Austrian Science Fund (FWF):
	M3211.
	
	\bibliographystyle{alpha}

\end{document}